\documentclass[11pt]{amsart}
\usepackage{amsmath,amsthm,amssymb,mathrsfs}
\usepackage{microtype,xspace}
\usepackage{cleveref,cite,url}
\usepackage{graphicx}
\usepackage{booktabs}
\usepackage[utf8]{inputenc}
\usepackage{enumitem}
\usepackage{xcolor}
\usepackage{longtable} 
\usepackage{algorithm}
\usepackage[noend]{algpseudocode}
\usepackage{float} 
\usepackage{tablefootnote}
\usepackage{subfigure}
\usepackage{multirow}
\usepackage{stmaryrd}
\usepackage{appendix}
\usepackage{arydshln}
\usepackage{soul}
\usepackage{tikz-cd}
\usetikzlibrary{shapes.geometric}
\tikzset{mymatr/.style={every outer matrix/.append style={draw=black, inner xsep=3pt , inner ysep=4pt, rounded corners, thick}}}

\graphicspath{ {../figures/} }
\makeatletter \def\input@path{ {../figures/} } \makeatother
\numberwithin{equation}{section}
\theoremstyle{plain}
\newtheorem{theorem}{Theorem}
\newtheorem{lemma}{Lemma}
\newtheorem{proposition}{Proposition}
\newtheorem{corollary}{Corollary}

\theoremstyle{definition}
\newtheorem{definition}{Definition}

\newtheorem{example}{Example}

\theoremstyle{remark}
\newtheorem*{remark}{Remark}

\newcommand{\RR}{\mathbb{R}}
\newcommand{\NN}{{\llbracket N \rrbracket}}
\renewcommand{\SS}{\mathbb{S}}

\newcommand{\FF}{\mathbb{F}}
\renewcommand{\subset}{\subseteq}

\DeclareMathOperator*{\rank}{rank}
\DeclareMathOperator*{\tr}{tr}
\DeclareMathOperator*{\diag}{diag}

\DeclareMathOperator{\sign}{sgn}
\let\vec\relax
\DeclareMathOperator{\vec}{vec}
\DeclareMathOperator{\vct}{vec}

\DeclareMathOperator{\polylog}{poly\hspace{.03em}log}

\begin{document}

\title[Rank-1 tensor completion with SDP]{Solving exact and noisy rank-one tensor completion with semidefinite programming}

\author{Diego Cifuentes}
\address{Georgia Institute of Technology \ Atlanta, GA, USA}
\email{diego.cifuentes@isye.gatech.edu}

\author{Zhuorui Li}
\address{Georgia Institute of Technology \ Atlanta, GA, USA}
\email{zli966@gatech.edu}

\maketitle

\begin{abstract}
Consider recovering a rank-one tensor of size \(n_1 \times \cdots \times n_d\) from exact or noisy observations of a few of its entries. We tackle this problem via semidefinite programming (SDP). We derive deterministic combinatorial conditions on the observation mask \(\Omega\) (the set of observed indices) under which our SDPs solve the exact completion and achieve robust recovery in the noisy regime. These conditions can be met with as few as \(\bigl(\sum_{i=1}^d n_i\bigr) - d + 1\) observations for special \( \Omega\).
When \(\Omega\) is uniformly random, our conditions hold with \(O\!\bigl(\sqrt{\prod_{i=1}^d n_i}\,\mathrm{polylog}(\prod_{i=1}^d n_i)\bigr)\) observations.
Prior works mostly focus on the uniformly random case, ignoring the practical relevance of structured masks. For \(d=2\) (matrix completion), our propagation condition holds if and only if the completion problem admits a unique solution. Our results apply to tensors of arbitrary order and cover both exact and noisy settings. In contrast to much of the literature, our guarantees rely solely on the combinatorial structure of the observation mask, without incoherence assumptions on the ground-truth tensor or uniform randomness of the samples.  Preliminary computational experiments show that our SDP methods solve tensor completion problems using significantly fewer observations than alternative methods.
\end{abstract}

\section{Introduction}
Tensors have emerged as fundamental tools in a wide range of applications in machine learning and signal processing \cite{morup2011applications,mossel2005learning,hsu2013learning,semerci2014tensor}.
Tensor completion problems are extensively investigated
to infer missing data from a limited number of observations 
\cite{li2017low,kreimer2013tensor,li2010tensor}.
While substantial progress has been made in the matrix completion case \cite{recht2010guaranteed,candes2012exact,candes2010power,candes2010matrix,keshavan2010matrix},
i.e., tensors of order two, extending these results to tensor completion of arbitrary order presents considerably greater challenges.

Rank-one tensor completion is a cornerstone for low-rank tensor completion because low-rank structure is built from rank-one components (see \cite{kolda2009tensor}). Robust recovery of a single rank-one factor provides sharp sample-complexity baselines, and well-conditioned subproblems that underlie greedy schemes that add components iteratively. Practically, accurate rank-one completion yields high-quality low-rank approximations.

Consider a real $d$-th order tensor $T \in \RR^{n_1 \times \dots \times n_d}$,
with entries $T_{i_1,\dots, i_d}$ indexed by tuples $(i_1,\dots,i_d)$
with $i_k \in [n_k] := \{1,\dots,n_k\}$.
The rank-one tensor completion problem consists in recovering the tensor given only (exact or noisy) entries corresponding to a  given observation mask $\Omega \subset [n_1] \times \dots \times [n_d]$.
The exact rank-one tensor completion problem can be decided in polynomial time by using linear algebra, as will be explained in \Cref{s:linearalgebra}.
The noisy rank-one tensor completion problem is more challenging.
Most past works assume that the observation mask is uniformly random and that the tensor values satisfy an incoherence property.
Under both assumptions, the noisy completion problem can be solved in polynomial time with $O\big(\max\big\{\sqrt{N} ,n\big\} \polylog n\big)$ random observations,
where
$N := n_1 n_2 \dots n_d$
and
$n := n_1 + \dots + n_d$,
see \cite{xia2021statistically,barak2016noisy,cai2019nonconvex,montanari2018spectral}.

For rank-one tensor completion, the minimal number of observations that guarantees uniqueness is \(n - d + 1\). This leads to a natural question: can stable recovery of a rank-one tensor be achieved efficiently with only \(n - d + 1\) observations?
In general, the answer is negative.
In this paper we identify deterministic combinatorial conditions on the observation mask \(\Omega\) under which stable recovery is indeed tractable.
Importantly, these deterministic structural conditions can be met with only \(n - d + 1\) observations.
When the observations are uniformly random,
our combinatorial conditions are met with 
\(O\!\left(\sqrt{N}\,\mathrm{polylog}\,N\right)\)
observations, matching existing results.

In this paper we tackle the exact and noisy rank-one tensor completion problems by means of convex optimization.
More precisely, we formulate tractable \emph{semidefinite programming} (SDP) relaxations for both exact and noisy settings.
Our SDPs involve a positive semidefinite (PSD) matrix of size
$(N + 1) \times(N + 1)$.
For exact completion, our SDP solves the problem when the observation mask \(\Omega\) satisfies the combinatorial conditions. In the noisy regime, our method is \emph{stable} under the same conditions, i.e., the reconstruction error tends to zero as the noise level vanishes.
To the best of our knowledge, our SDP is the first tractable method that guarantees stable recovery for general \(d\)-th order asymmetric rank-one tensors.

Our methods allow to tackle low-rank tensor completion problems using a greedy, iterative framework in which several rank-one components are recovered one after another. We provide preliminary experiments illustrating the performance of our low-rank tensor completion method on images of size $256 \times 256 \times 3$
(196000 entries).

In contrast to much of the prior literature, we neither impose incoherence assumptions on the unknown tensor nor assume that the observations are drawn uniformly at random. Instead, we solely rely on deterministic, combinatorial conditions on the observation mask. This is practically relevant since in many practical applications the observation mask is not uniformly random. As a representative example, in image or video inpainting the observation mask typically consists of axis-aligned slices. For image inpainting the tensor has size $n_1 \times n_2 \times 3$, and the mask consists of slices along the third axis. These slices are often spatially clustered rather than uniformly random. In this setting, our deterministic conditions specialize to the matrix case and are satisfied whenever the completion problem admits a unique solution. 

\subsection*{Contributions}
First, we characterize conditions on the observation mask \(\Omega\) under which the exact rank-one tensor completion problem admits a \emph{unique} solution, independently of the ground truth \(\hat T\). We then introduce four sufficient \emph{propagation} conditions on \(\Omega\), which satisfy the following relations:
\begin{equation}\label{eq:conditions}
\begin{tikzcd}[sep=small,
mymatr,every arrow/.append style = {-stealth, shorten > = 2pt, shorten <=2pt},
                    ]
\text{\footnotesize A-propaga.} \ar[r]&\text{\footnotesize SR-propaga.} \ar[r]&\text{\footnotesize S-propaga.} \ar[r]&  \text{\footnotesize GS-propaga.} \ar[r] & \text{\footnotesize Unique recovery}           
    \end{tikzcd}
\end{equation}

For \(d=2\) (matrix completion), all propagation conditions are equivalent to unique recovery. For general \(d\), the minimal number of observations required to satisfy any of the propagation conditions coincides with the minimal number that guarantees unique recovery. Moreover, although our criteria are fully deterministic, when \(\Omega\) is sampled uniformly at random all propagation conditions hold with high probability once the sample size is on the order of \(\sqrt{N}\log N\), matching the best-known results.

We propose a SDP relaxation, \eqref{eq:sdp2}, for exact tensor completion. We prove that \eqref{eq:sdp2} exactly recovers the tensor under the SR-propagation condition. In particular, when \(d=2\), our SDP succeeds for every observation mask \(\Omega\) that guarantees unique recovery. We further develop modifications of the SDP to accommodate the weaker S-propagation condition.
We introduce a second SDP relaxation, \eqref{eq:sdp-p}, tailored to noisy tensor completion. Under the SR-propagation assumption, \eqref{eq:sdp-p} reconstructs all entries of the unknown rank-one tensor with error bounded by \(O\!\bigl(\sqrt{\|\epsilon\|_\infty}\bigr)\). Finally, we present preliminary computational experiments to illustrate the performance of our SDP formulations. Using off-the-shelf solvers, we evaluate small-scale instances, and for medium-scale tensors we employ a custom SDP solver based on hierarchical low-rank decompositions. The experiments show that our methods allow to tackle rank-one completion problems using significantly less observations than previous methods. 
We also provide experiments on low-rank completion problems using a greedy, iterative framework.

\subsection*{Related works}

The tensor completion problem can be converted into matrix completion by unfolding the tensor.
Hence, it is possible to rely on techniques from matrix completion,
see, e.g., \cite{kolda2009tensor,gandy2011tensor,huang2015provable,liu2012tensor,tomioka2010estimation,goldfarb2014robust,romera2013new}.
However, ignoring the tensor structure often leads to a larger number of required observations,
e.g., $O(n^{2} \polylog n)$ observations instead of $O(n^{3/2} \polylog n)$ for third-order tensors of size $n\times n \times n$,
see \cite{yuan2016tensor,yuan2017incoherent}.

Several papers have studied methods for tensor completion.
Most papers assume that the observations are uniformly random and that the (unknown) completed tensor satisfies some type of incoherence property.
The papers \cite{yuan2016tensor, yuan2017incoherent} propose intractable methods that only need $O(n^{3/2}\polylog n)$ random observations for tensors of any order.
Tractable methods using $O(\sqrt{N}\polylog n)$ random observations are proposed on \cite{xia2019polynomial, liu2020tensor, barak2016noisy, potechin2017exact} for third-order tensors,
and on \cite{xia2021statistically} for arbitrary order tensors.
In contrast to all these works,
we do not assume that the samples are uniformly random and we do not make any assumption about the unknown completed tensor.
Instead, we only assume that the observed values are nonzero and we study combinatorial conditions on the observation mask that guarantee recovery.

We discuss \cite{barak2016noisy,potechin2017exact} in more detail since they also use SDP.
Both papers focus on third-order tensors, require $O(\sqrt{N} \polylog n)$ uniformly random samples, and make incoherence assumptions on the true tensor values.
In contrast, our methods apply to arbitrary order tensors and work under deterministic combinatorial conditions on the mask,
instead of incoherence assumptions on the true tensor.
Regarding Barak and Moitra \cite{barak2016noisy}, they focus exclusively on the noisy completion problem,
and propose an SDP with a PSD matrix of size ${3n {+} 3 \choose 3}\times {3n {+} 3 \choose 3}$, much larger than ours which requires significantly less scalar optimization variables. Moreover, the error bound established in \cite{barak2016noisy} does not vanish (is a positive constant) even the noise tends to zero. Next,  Potechin and Steurer \cite{potechin2017exact} focus exclusively on exact completion,
and provide results for the same as SDP from \cite{barak2016noisy} and also for a smaller SDP,
with a PSD matrix of size ${n(n{+}1)}\times {n(n{+}1)}$.
However, our experimental results show that their smaller SDP behaves quite poorly in practice.

The closest work to this paper is that of Cosse and Demanet \cite{cosse2021stable},
since they do not assume incoherence nor uniformly random observations.
However, most of their results focus on rank-one matrix completion ($d\!=\!2$).
For matrix completion, they consider an SDP with matrix size approximately $\tfrac{1}{2}(n_1 {+} n_2)^2 \times \tfrac{1}{2}(n_1 {+} n_2)^2$,
larger than ours.
They show that it solves the exact rank-one matrix completion problem whenever there is a unique solution.
Our relaxation \eqref{eq:sdp2} has the same guarantees.
They also consider a modified SDP for the noisy matrix completion problem,
and they provide guarantees similar to those of~\eqref{eq:sdp-p}.
An inconvenience of their SDP is that it requires knowledge of a bound on the noise magnitude, and the solution is sensitive to the given bound.
In contrast, our \eqref{eq:sdp-p} attempts to minimize the error with respect to the observed values,
so no knowledge about the noise is required.

Cosse and Demanet \cite{cosse2021stable} provide a few results for tensors of order $d\geq 3$.
They show that an SDP with matrix size $O(n^{d+1}) \times O(n^{d+1})$
solves exact tensor completion problem under the A-propagation condition.
In contrast, our SDP is significantly smaller,
with matrix size $(N {+} 1) \times (N {+} 1)$,
and our guarantees hold under the more general SR-propagation and S-propagation conditions.
No results for noisy tensor completion are provided in \cite{cosse2021stable}.

Two recent papers consider rank-one tensor completion. Zhou et al.\ \cite{zhou2025rank} reformulate the noiseless third-order problem as a rank-one matrix completion task, and Nie et al.\ \cite{nie2025robust} extend this line of work to noisy observations via a biquadratic model. Because the formulated problems are computationally challenging, both papers introduce semidefinite relaxations. Zhou’s guarantees apply to a hierarchy rather than a single SDP, and Nie’s tightness holds only under a restrictive separability assumption with limited practical relevance. Neither paper provides \(\Omega\)-based performance guarantees for their SDPs. In contrast, we address the general \(d\)th-order rank-one tensor completion problem and provide verifiable conditions on \(\Omega\) under which our SDP achieves exact recovery in the noiseless regime and stable recovery in the presence of noise.

\subsection*{Notation}

Throughout this paper we assume that $n_1, \dots , n_d$ are fixed positive integers.
Let $n := n_1 + \dots + n_d$ and $N := n_1 n_2 \dots n_d$ be their sum and product.
For integers $\ell,m$,
let $[m] := \{1,\dots, m\}$
and $[\ell, m] := \{\ell,\dots, m\}$.
We also let $\NN:= [n_1]\times\dots\times [n_d]$.
We use Greek letter to denote tuples of indices, e.g., $\alpha = (i_1,\dots,i_d) \in \NN$.
Let $\RR^n$ be the space of real vectors of length~$n$,
let $\RR^{n_1\times \dots \times n_d} = \RR^{\NN}$ be the space of real tensors of size $n_1\times \dots \times n_d$,
and let $\SS^k$ be the space of $k \times k$ real symmetric matrices.
We also let $\FF_2$ be the finite field with two elements.
Given $X \in \SS^k$, the notation $X \succeq 0$ means that $X$ is PSD.
Given a vector $w \in \RR^k$, we let $\diag(w) \in \SS^k$ be the diagonal matrix with $w$ in its diagonal.

\section{Rank-one tensor completion and unique recovery}
\label{s:linearalgebra}

In this section, we formalize the rank-one tensor completion problem and present a linear algebra characterization for determining whether a given observation mask admits a unique recovery.

A real $d$-th order tensor $T \in \RR^{n_1 \times \dots \times n_d}$  has \emph{rank one} if it is the outer product
$T = u^{(1)} \otimes \dots \otimes u^{(d)}$ of some vectors 
$u^{(1)} \in \RR^{n_1},\dots, u^{(d)}\in \RR^{n_d}$,
i.e.,
\begin{equation}\label{eq:tensor_representation}
  T_{i_1,\dots,i_d} = \prod_{k=1}^d u^{(k)}_{i_k},
  \qquad \forall (i_1,\dots,i_k) \in \NN.
\end{equation}
We highlight that there are two different notions of rank for tensors, Tucker and Canonical Polyadic (CP),
but both notions agree in the special case of rank one tensors.

The exact rank-one \emph{tensor completion} problem consists in determining a rank-one tensor~$T$, given the (exact) values of a subset of its entries.
More precisely,
given an observation mask $\Omega \subset \NN$ and  a vector of observed values $\hat{T} \in \RR^\Omega$, the problem is:
\begin{equation}\label{eq:exactcompletion}
\begin{aligned}
  \text{ find } \quad &T\in \RR^{n_1 \times \dots \times n_d}\\
  \quad\text{ s.t. }\quad
  &T_{i_1,\dots,i_d} = \hat{T}_{i_1,\dots,i_d} \;\forall (i_1,\dots,i_d) \in \Omega
  \quad\text{ and }\quad
  \rank T = 1
\end{aligned}
\end{equation}

We assume that all entries of $\hat{T}$ are nonzero, as otherwise it is possible to remove an slice of $T$ and obtain an equivalent completion problem of smaller dimension.

We also consider the \emph{noisy} setting,
where the observed values are corrupted by noise.
More precisely, given a noise vector $\epsilon \in \RR^\Omega$, let
\[
  T^\epsilon_{i_1,\dots,i_d} := \hat{T}_{i_1,\dots,i_d} + \epsilon_{i_1,\dots,i_d},
  \qquad \forall (i_1,\dots,i_k) \in \Omega.
\]
be the vector of corrupted observations.
For noisy rank-one tensor completion we are given $\Omega, T^\epsilon$, and the problem is:
\begin{equation}\label{eq:noisycompletion}
\begin{aligned}
  \text{ find } \quad &T\in \RR^{n_1 \times \dots \times n_d}\\
  \quad\text{ s.t. }\quad
  &T_{i_1,\dots,i_d} \approx T^\epsilon_{i_1,\dots,i_d} \;\forall (i_1,\dots,i_d) \in \Omega
  \quad\text{ and }\quad
  \rank T = 1
\end{aligned}
\end{equation}
where the approximation errors depend on the the noise level~$\epsilon$.

For exact completion \eqref{eq:exactcompletion}, we present our SDP in \Cref{section:sdp-e}; for noisy completion \eqref{eq:noisycompletion}, the corresponding SDP is discussed in \Cref{section:sdp-n}. Before introducing these formulations, we first analyze the unique recovery property of the observation mask \(\Omega\). In \Cref{section:propagation_conditions}, we then characterize structural conditions on \(\Omega\) that yield theoretical guarantees for our methods.

An observation mask $\Omega \subset \NN$ satisfies \emph{unique recovery} if for every $\hat T \in \RR^\Omega$ with nonzero entries the completion problem \eqref{eq:exactcompletion} is either infeasible or has a unique solution.

The unique recovery property can be identified in polynomial time by using linear algebra;
see~\cite{singh2020rank}.
For completeness, we explain here how to do this. The basic idea is to reduce the tensor completion problem to solving two systems of linear equations;
one system determines the absolute value of the entries
and the other one determines the signs.

Let us see that an observation mask $\Omega \subset \NN$ can be identified with a matrix ${V}_{\Omega}^{\FF_2}\in\FF_2^{|\Omega|\times n}$ with binary entries,
where $n := n_1 {+} \dots {+} n_d$.
An element $j \in [n_\ell]$ can be identified with the indicator vector $e^{(n_\ell)}_j \in \FF_2^{n_\ell}$,
with 1 in the $j$-th entry and 0's everywhere else.
Hence, a tuple $\alpha = (j_1,\dots,j_d)\in \NN$
can be identified with the vector
\begin{equation}\label{ind_vec}
    v_\alpha = v_{j_1,\dots ,j_d} :=
    \begin{pmatrix}
    e^{(n_1)}_{j_1} \\ \vdots \\ e^{(n_d)}_{j_d}
    \end{pmatrix}
    \in \FF_2^{n},
\end{equation}
which concatenates the~$e^{(n_\ell)}_{j_\ell}$'s.
Hence, $\Omega$ can be identified with the matrix
\begin{align}\label{ind_mat}
V_{\Omega}^{\FF_2} =
\begin{pmatrix}
  \vdots\\
  \text{\textemdash}\;\; v_{\alpha}^T\; \text{\textemdash} \\
  \vdots
\end{pmatrix}
\in\FF_2^{|\Omega|\times n}
\end{align}
with rows $v_\alpha \in \FF_2^n$ for all~$\alpha \in \Omega$.

We denote ${V}^{\RR}_{\Omega}\in\RR^{|\Omega|\times n}$ the matrix with the same entries as ${V}_{\Omega}^{\FF_2}$, but viewed as real numbers.
We are ready to introduce the two systems of linear equations.
Given the mask $\Omega$ and observations $\hat T \in \RR^\Omega$,
our goal is to find vectors $u^{(1)}\in \RR^{n_1}, \dots, u^{(d)}\in \RR^{n_d}$ such that
\begin{equation}\label{eq:tensor_representation2}
  \prod_{k=1}^d u^{(k)}_{i_k} = \hat{T}_{i_1,\dots,i_d},
  \qquad \forall (i_1,\dots,i_d) \in \Omega.
\end{equation}
Taking the log of absolute values on both of sides, we get
\begin{equation*}
    \sum_{k=1}^d \log(|u^{(k)}_{i_k}|) = \log(|\hat T_{i_1,\dots,i_d}|),
    \qquad\forall (i_1,\dots,i_d)\in \Omega.
\end{equation*}
The above is a system of linear equations in variables $y_{k, i_k} = \log(|u^{(k)}_{i_k}|)$,
which might be written as
\begin{equation}\label{system:abs}
  V^\RR_{\Omega} \; y \;=\; \log(|\hat T|)
\end{equation}

The solution of~\eqref{system:abs} gives the absolute values of the entries of the $u^{(k)}$'s.
It remains to explain how to recover the signs.
Consider the function
\begin{equation*}
  \sign : \RR \setminus \{0\} \to \FF_2,
  \qquad
  t \mapsto \begin{cases}
        0, \quad\text{if}~t > 0\\
        1, \quad\text{if}~t < 0
    \end{cases}
\end{equation*}
Applying this function in both sides of \eqref{eq:tensor_representation2} leads to
\begin{equation*}
    \sum_{k=1}^d \sign(u^{(k)}_{i_k}) \underset{(\FF_2)}{=} \sign(\hat T_{i_1,\dots,i_d}),
    \qquad\forall (i_1,\dots,i_d)\in \Omega.
\end{equation*}
The above is a linear system over $\FF_2$ in variables $z_{k, i_k} = \sign(u^{(k)}_{i_k})$,
which might be written in the form
\begin{equation}\label{system:signs}
  V^{\FF_2}_{\Omega} \; z \;=\; \sign(\hat T)
\end{equation}

The solution of \eqref{system:signs} determines the signs of the entries of the $u^{(k)}$'s.

Given $\Omega, \hat T$, the completion problem has a unique solution if and only if the systems of equations in \eqref{system:abs} and \eqref{system:signs} have unique solutions, up to trivial equivalences.
More precisely, we have to account for the fact that we may scale the different vectors $u^{(k)}$'s without changing the underlying tensor.
This leads to the following result.

\begin{proposition}\label{prop:rank_uniqueness}
    A set $\Omega \subset \NN$ satisfies unique recovery if and only if
    the matrix $V_{\Omega}^{\FF_2}$ from \eqref{ind_mat} has rank $n-d + 1$ (over $\FF_2$).
\end{proposition}
\begin{proof}
  From the above discussion we have that unique recovery holds if and only if the ranks of the matrices $V^{\FF_2}_{\Omega}$ and $V^{\RR}_{\Omega}$ match the number of degrees of freedom in the variables $u^{(1)},\dots,u^{(d)}$,
  modulo rescaling.
  Since we may assume that the first entry of each of the vectors $u^{(2)}, \dots, u^{(d)}$ is a 1,
  the number of degrees of freedom is $n-d+1$.
  The proposition now follows from the fact that
  $\rank(V^{\RR}_{\Omega}) \geq \rank(V^{\FF_2}_{\Omega})$.
\end{proof}

\begin{example}\label{example_linear_system}
  Let $\NN = [2]\times [2]\times[2]$ and $\Omega=\{(1,1,1), (1,2,2), (2,1,2), (2,2,2)\}$.
  The matrix associated to $\Omega$ is
  \[V_{\Omega}^{\FF_2}=
    \begin{bmatrix}
      1 &0 &1 &0 &1 &0\\
      1 &0 &0 &1 &0 &1\\
      0 &1 &1 &0 &0 &1\\
      0 &1 &0 &1 &0 &1\\
    \end{bmatrix}.
  \]
  Since $ \rank(V^{\FF_2}_{\Omega}) = 4 = n - d + 1$,
  then $\Omega$ satisfies unique recovery.
  In particular, if the observed values are 
  \[
    \hat{T}_{111} = 1,~~ \hat{T}_{122} = -4,~~  \hat{T}_{212} = -4,~~ \hat{T}_{222} = -8,
  \]
  then by solving the systems of equations \eqref{system:abs} and \eqref{system:signs}
  we can recover the remaining entries:
  \[
    T_{121} = 2,~~ T_{211} = 2,~~  T_{221} = 4,~~ T_{112} = -2.
  \]
\end{example}
Note that if the observed values $\hat T$ are corrupted by noise then the linear system in \eqref{system:abs} will generally have no solution.
Thus, the recovery method from above does not apply in noisy cases.

\section{Sufficient conditions for unique tensor recovery}\label{section:propagation_conditions}

In this section we study sufficient conditions under which the exact rank-one tensor completion problem \eqref{eq:exactcompletion} has a unique solution.
We propose four deterministic combinatorial conditions:
A-propagation,
SR-propagation,
S-propagation,
and GS-propagation.
The conditions satisfy the relations in \eqref{eq:conditions},
and they are used to analyze the guarantees of our SDPs.

As an illustration, \Cref{tab:conditions} reports the fraction of observation masks \(\Omega\) that satisfy each condition for tensors of sizes $[3]{\times}[3]{\times}[2]$ and $[2]{\times}[2]{\times}[2]{\times}[2]$.

Notably, the fraction of instances satisfying A-propagation declines sharply when we restrict to masks with the minimal number of required observations.

\begin{table}[htb]
  \centering
  \caption{Percentage of masks $\Omega$ satisfying each condition.}
  \label{tab:conditions}
  \begin{tabular}{ccccc}
    \toprule
    & \multicolumn{2}{c}{$[3]{\times}[3]{\times}[2]$}
    & \multicolumn{2}{c}{$[2]{\times}[2]{\times}[2]{\times}[2]$} \\
    & $|\Omega|\!=\!6$ & $|\Omega|\!\leq\!18$
    & $|\Omega|\!=5$ & $|\Omega|\!\leq\!16$\\
    \midrule
    Unique Recovery& 48.35\% & 86.76\% & 61.54\% &91.90\% \\
    GS-propagation & 48.35\% & 86.76\% & 61.54\% &91.90\% \\
    S-propagation  & 48.29\% & 86.76\% & 61.17\% &91.88\% \\
    SR-propagation & 47.90\% & 86.73\% & 61.17\% &91.88\% \\
    A-propagation  & \phantom{0}9.31\%  & 72.51\% & \phantom{0}9.16\% &70.60\% \\
    \bottomrule
  \end{tabular}
\end{table}

This section is organized into three parts.
First, we describe each of the four sufficient conditions.
Then we show the relationships among those conditions. 
Finally, we characterize how many observations are needed to satisfy these conditions when $\Omega$ is uniformly random.

\subsection{Propagation conditions}

We proceed to introduce the four propagation conditions from \eqref{eq:conditions}.

\subsubsection{Generalized Square (GS) propagation}

A \emph{generalized square} is a set $\{\alpha_1, \alpha_2,\alpha_3,\alpha_4\}$ of four tuples $\alpha_i \in \NN$ such that
\begin{equation}\label{generalized_square}
  v_{\alpha_1} + v_{\alpha_2} + v_{\alpha_3} + v_{\alpha_4} \underset{(\FF_2)}{=} 0,
\end{equation}
where $v_{\alpha_i}$ is as in~\eqref{ind_vec}.
An illustration of a generalized square is shown on the right of \Cref{fig:square}.

\begin{figure}[htb]
    \centering
    \includegraphics[width=0.3\textwidth]{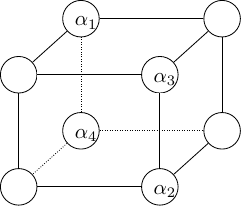} \qquad\qquad
    \includegraphics[width=0.3\textwidth]{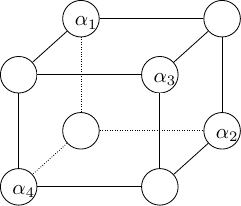}
    \caption{A square (left) and a generalized square (right).}
    \label{fig:square}
\end{figure}

The GS-propagated set $\mathcal{P}_{GS}(\Omega)$ of $\Omega$ is defined recursively as follows:
\begin{itemize}
  \item $\Omega \subset \mathcal{P}_{GS}(\Omega)$.
  \item If $\{\alpha_{1}, \alpha_{2}, \alpha_{3}, \beta \}$ is a GS and each $\alpha_i \in \mathcal{P}_{GS}(\Omega)$, then $\beta \in \mathcal{P}_{GS}(\Omega)$.
\end{itemize}
The set $\Omega$ satisfies \emph{GS-propagation} if $\mathcal{P}_{GS}(\Omega) = \NN$.

\begin{example}[GS holds]\label{eg:GSholds}
    Let $\NN = [3] \times [3] \times [2]$ and 
    \[
    \Omega = \{(1,1,1),(1,2,2),(2,1,2),(3,2,1),(2,3,1),(3,3,2)\}.
    \]
    The GS-propagation holds in this case.
    As an illustration, we verify that (2,2,1), (2,2,2) and (3,2,2) belong to $\mathcal{P}_{GS}(\Omega)$:
    \begin{align*}
        \{(1,1,1), (2,1,2), (1,2,2), (2,2,1)\}
        \text{ is a GS }
        \quad\implies\quad
        (2,2,1) \in \mathcal{P}_{GS}(\Omega),
        \\
        \{(3,2,1), (2,3,1), (3,3,2), (2,2,2)\}
        \text{ is a GS }
        \quad\implies\quad
        (2,2,2) \in \mathcal{P}_{GS}(\Omega),\\
        \{(3,2,1), (2,2,1), (2,2,2), (3,2,2)\}
        \text{ is a GS }
        \quad\implies\quad
        (3,2,2) \in \mathcal{P}_{GS}(\Omega).
    \end{align*}
\end{example}

The next example illustrates that GS-propagation is not necessary for unique recovery.

\begin{example}[GS fails]\label{eg:GSfails}
  Let $\NN=[6]\times [6] \times [6]$ and
  \begin{align*}
    \Omega =& \{(1,1,1), (2,2,1), (3,3,1),(4,1,2),(3,4,2),(2,5,2),(5,2,3),(6,3,3),\\
    &(1,4,3),(4,2,4),(1,5,4),(3,6,4),(2,6,5),(5,1,5),(4,3,6),(5,5,6)\},
  \end{align*}
  It can be checked that unique recovery holds by \Cref{prop:rank_uniqueness}.
  On the other hand, GS-propagation condition fails since we have that $\mathcal{P}_{GS}(\Omega) = \Omega$. 
\end{example}

\subsubsection{Square (S) propagation}
A \emph{square} consists of four distinct tuples $\alpha_1, \alpha_2 , \alpha_3,\alpha_4 \!\in\! \NN$ such that
\begin{equation}\label{general_square}
  \begin{aligned}
    \qquad\max\{(\alpha_1)_k,(\alpha_2)_k\} &= \max\{(\alpha_3)_k,(\alpha_4)_k\}, &\quad \forall~~k\in[d]\qquad\qquad\\
    \qquad\min\{(\alpha_1)_k,(\alpha_2)_k\} &= \min\{(\alpha_3)_k,(\alpha_4)_k\},&\quad \forall~~k\in[d]\qquad\qquad
  \end{aligned}
\end{equation}
The pairs $\alpha_1, \alpha_2$ and $\alpha_2,\alpha_4$ have \emph{opposite} elements,
while the other four pairs have \emph{adjacent} elements.
We denote a square as $\{\alpha_1, \alpha_2 ; \alpha_3, \alpha_4\}$

An illustration of a square is shown on the left of \Cref{fig:square}.
For instance,
for $\NN = [2] \times [2] \times [2]$
there are 12 squares:
\begin{equation}\label{eq:squares12}
\begin{aligned}
  \{111,222;112,221\}, \quad
  \{111,222;211,122\}, \quad
  \{111,222;121,212\},
  \\
  \{112,221;211,122\}, \quad
  \{112,221;121,212\}, \quad
  \{211,122;121,212\},
  \\
  \{111,122;112,121\}, \quad
  \{111,212;112,211\}, \quad
  \{111,221;121,211\},
  \\
  \{211,222;212,221\}, \quad
  \{121,222;122,221\}, \quad
  \{112,222;122,212\}.
\end{aligned}
\end{equation}

The S-propagated set $\mathcal{P}_{S}(\Omega)$ of $\Omega$ is defined recursively as follows:
\begin{itemize}
  \item $\Omega \subset \mathcal{P}_{S}(\Omega)$.
  \item If $\{\alpha_{1}, \alpha_{2}; \alpha_{3}, \beta \}$ is a square and each $\alpha_i \in \mathcal{P}_{S}(\Omega)$, then $\beta \in \mathcal{P}_{S}(\Omega)$.
\end{itemize}
The set $\Omega$ satisfies
S-propagation if $\mathcal{P}_{S}(\Omega) = \NN$.

\begin{example}[S holds]\label{eg:Sholds}
Let $\NN = [3]\times[4]\times[3]$ with observation set 
\[\Omega=\{(1,1,1), (1,2,2), (2,1,2), (2,2,2), (3,3,2), (3,1,3),(3,3,3), (2,4,3)\}.\]
The S-propagation condition holds in this case.
As an illustration, we verify that (3,4,1) belongs to $\mathcal{P}_{S}(\Omega)$:
\begin{align*}
    \{(1,1,1), (2,1,2), (2,2,2), (1,2,1)\}
    \text{ is a square }
    \quad\implies\quad
    (1,2,1) \in \mathcal{P}_{S}(\Omega),
    \\
    \{(3,3,2), (3,3,3), (3,1,3), (3,1,2) \}
    \text{ is a square }
    \quad\implies\quad
    (3,1,2) \in \mathcal{P}_{S}(\Omega),\\
    \{(2,1,2), (2,2,2), (3,1,2), (3,2,2) \}
    \text{ is a square }
    \quad\implies\quad
    (3,2,2) \in \mathcal{P}_{S}(\Omega),\\
    \{(1,2,2), (1,2,1), (3,2,2), (3,2,1) \}
    \text{ is a square }
    \quad\implies\quad
    (3,2,1) \in \mathcal{P}_{S}(\Omega),\\
    \{(2,4,3), (3,3,3), (3,3,2), (2,4,2) \}
    \text{ is a square }
    \quad\implies\quad
    (2,4,2) \in \mathcal{P}_{S}(\Omega),\\
    \{(2,2,2), (2,4,2), (3,2,1), (3,4,1) \}
    \text{ is a square }
    \quad\implies\quad
    (3,4,1) \in \mathcal{P}_{S}(\Omega).
\end{align*}
\end{example}

\begin{example}[S fails]\label{eg:Sfails}
Consider the same $\Omega$ as in \Cref{eg:GSholds}.
S-propagation is not satisfied since $\mathcal{P}_{S}(\Omega) = \Omega$.
\end{example}

\subsubsection{Square Restricted (SR) propagation}

The SR-propagated set $\mathcal{P}_{SR}(\Omega)$ of a mask $\Omega$ is defined recursively as follows:
\begin{itemize}
  \item $\Omega \subset \mathcal{P}_{SR}(\Omega)$.
  \item If $\{\alpha_{1}, \beta; \alpha_{2}, \gamma \}$ is a square,
    $\alpha_1, \alpha_2 \!\in\! \Omega$,
    and $\beta \!\in\! \mathcal{P}_{SR}(\Omega)$,
    then $\gamma \!\in\! \mathcal{P}_{SR}(\Omega)$.
\end{itemize}
The set $\Omega$ satisfies
SR-propagation if $\mathcal{P}_{SR}(\Omega) = \NN$.

\begin{example}[SR holds]\label{eg:SRholds}
Let $\Omega$ be the same as in Example~\ref{example_linear_system}.
We verify that $\gamma \in \mathcal{P}_{SR}(\Omega)$ for each $\gamma \in \NN\setminus \Omega$:
\begin{align*}
    \{(1,2,2), (2,1,2); (2,2,2), (1,1,2)\}
        \text{ is a square }
        \quad\implies\quad
        (1,1,2) \in \mathcal{P}_{SR}(\Omega),
        \\
        \{(1,1,1), (2,2,2); (1,2,2), (2,1,1) \}
        \text{ is a square }
        \quad\implies\quad
        (2,1,1) \in \mathcal{P}_{SR}(\Omega),
        \\
        \{(1,1,1), (2,2,2); (2,1,2), (1,2,1) \}
        \text{ is a square }
        \quad\implies\quad
        (1,2,1) \in \mathcal{P}_{SR}(\Omega),
        \\
        \{(2,2,2), (1,2,1); (1,2,2), (2,2,1) \}
        \text{ is a square }
        \quad\implies\quad
        (2,2,1) \in \mathcal{P}_{SR}(\Omega).
\end{align*}
Hence, the SR-propagation condition holds.
\end{example}

\begin{example}[SR fails]\label{eg:SRfails}
    In the instance from \Cref{eg:Sholds},
    it can be checked that $\mathcal{P}_{SR}(\Omega) = \NN\setminus\{(3,4,1)\}$. 
    Hence, SR-propagation is not satisfied.
\end{example}

\subsubsection{Axes (A) propagation}

A sequence $\alpha_1,\alpha_2,\dots,\alpha_p \in \NN$ is \emph{strongly connected}
\footnote{
  The name is motivated by the fact that it corresponds to a strong notion of connectivity for the hypergraph with vertex set $\bigcup_k [n_k]$ and hyperedges $\{\alpha_k\}_k$.
}
if for each $2\leq k\leq p$ we can find an $1\leq i \leq k-1$
such $\alpha_k, \alpha_i$ have $d-1$ indices in common.
A set $\Omega \subset \NN$ satisfies the \emph{A-propagation} condition if we can find a strongly connected sequence
\(\alpha_1,\alpha_2,\dots,\alpha_p \in \Omega,\)
such that $\{\alpha_i\}_{i\in [p]}$ covers $\NN$ coordinate-wise, i.e., 
\[
  \bigcup_{k=1}^{p}(\alpha_k)_i = [n_i],\quad\text{for}~i=1,2\dots,d
\]
where $(\alpha_k)_i$ is the $i$-th coordinate of $\alpha_k$.

\begin{remark}
Cosse and Demanet show in \cite[Cor.~2]{cosse2021stable} that the A-propagation condition allows the exact rank-one tensor approximation problem to be solved by a certain SDP relaxation (one larger than ours).
We point out that statement of \cite[Cor.~2]{cosse2021stable} mentions a weaker condition,
but the proof in fact uses A-propagation
\footnote{A-propagation is used in equations (5.2), (5.10), (5.11), (5.12), (5.13) for the proof of \cite[Cor.~2]{cosse2021stable}.
For instance, in (5.2), they construct a sequence of polynomial equations whose indices adhere to A-propagation.
}.
\end{remark}

This is an example where axes propagation condition holds.

\begin{example}[A holds]\label{eg:Aholds}
  Let $\NN = [2] \times [2] \times [2]$ and \[
  \Omega=\{(1,1,1),(1,2,1),(1,2,2),(2,2,2)\}.\]
    The axes propagation condition is verified using the sequence
    \[
    (1,1,1)\rightarrow(1,2,1)\rightarrow(1,2,2)\rightarrow(2,2,2).
    \]
\end{example}

\begin{example}[A fails]\label{eg:Afails}
Consider the instance from \Cref{eg:SRholds},
which satisfies SR-propagation.
The axes propagation condition fails.
The reason is that for $(1,1,1)\in \Omega$, all of the other elements in $\Omega$ have at least 2 different indices with $(1,1,1)$ and then the third coordinate cannot be fully covered.
\end{example}

\subsection{Relationships among propagation conditions}

We are ready to prove the chain of implications in~\eqref{eq:conditions}.

\begin{theorem}\label{thm:condition_equivalence}
  The propagation conditions satisfy~\eqref{eq:conditions}.
  In addition, the implications are all strict for $d\geq 3$.
\end{theorem}
\begin{proof}
  The instances showing that the implications are strict are given in \Cref{eg:GSfails,eg:Sfails,eg:SRfails,eg:Afails}.
  We proceed to prove each of the implications:
  \begin{itemize}
    \item (GS $\implies$ Unique recovery)
      Consider the vector space $L \subset \FF_2^n$ spanned by the vectors $v_\alpha$, for $\alpha \in \Omega$.
      The definition of GS-propagation
      implies that for every $\beta \in \NN$ the vector $v_\beta$ lies in the $L$.
      It follows that $\rank(V^{\FF_2}_{\Omega}) = n-d+1$, and hence unique recovery holds.
    \item (S $\implies$ GS)
      Follows from a square also being a generalized square.
    \item (SR $\implies$ S)
      Follows directly from the definitions.
    \item (A $\implies$ SR)
      This part of the proof appears in \Cref{append:prop:AS-US}.
      \qedhere
  \end{itemize}
\end{proof}

For the special case $d=2$ (corresponding to matrix completion), the propagation conditions are all equivalent to unique recovery, as stated next.

\begin{proposition}\label{matrix_condition_equivalence}
    If $d\!=\!2$,
    the four propagation conditions are equivalent to unique recovery.
    Moreover, unique recovery is equivalent to
    the connectedness of the bipartite graph with vertex set $V= \{u_1, \dots, u_{n_1}\} \cup \{v_1, \dots, v_{n_2}\}$,
    where $(u_i, v_j)$ are connected if and only if $(i,j) \in\nobreak \Omega$.
\end{proposition}

\begin{proof}
  The equivalence of A-propagation and unique recovery for $d=2$ is proved in \cite[Lem.~1]{cosse2021stable}.
  Together with the chain of implications in~\eqref{eq:conditions},
  we conclude that all the propagation conditions are equivalent for $d=2$.
  The bipartite graph characterization of unique recovery can be found in \cite{kiraly2015algebraic}.
\end{proof}

We point out that each of the propagation conditions leads to a simple sequential algorithm for exact rank-one tensor completion.
For instance, in the case of GS-propagation, if $\{\alpha_1,\alpha_2,\alpha_3,\beta\}$ is a GS and the entries of the tensor in positions $\alpha_1,\alpha_2,\alpha_3$ are known, then it is easy to complete the entry in position~$\beta$,
and we can keep completing entries following the propagation path.
However, these sequential algorithms are highly unstable in the presence of noise.
Indeed, \cite[\S1.1]{cosse2021stable} illustrates that the presence of noise leads to errors in the recovered entries that increase exponentially in the length of the propagated sequence.

\subsection{Propagation conditions for random observations}

We next show that A-propagation (the strongest propagation condition) is satisfied with high probability when the observations are uniformly random and are on the order of $O(\sqrt{N}\polylog n)$.

\begin{theorem}\label{general_hp}
  For any $d\in \mathbb{Z}_+$, and $3\leq n_1\leq n_2\leq\dots\leq n_d$,
  such that
  $n_{i-1}/\sqrt{n_i}\geq \tfrac{2}{i-1}$ for $i\geq 3$.
  Let $\Omega\subseteq \NN$ be a random subset, so that each $(i_1,i_2,\dots,i_d)$ lies in $\Omega$, independently with probability $p$, where, with a constant $C\geq 2 \prod_{i=3}^d \big((1 - n_i^{-\sqrt{n_i}})^{-1}\big)$, 
\begin{equation*}
    p \geq
    \frac{C}{\sqrt{N}} \prod_{i=2}^d \log(n_i)
\end{equation*}
Then $\Omega$ satisfies A-propagation with probability at least $1-\sum_{i=2}^{d} \varepsilon_i$, where

\small{\begin{equation}\label{bound whp}
    \varepsilon_i = 
    \begin{cases}
    2\max \big\{n_2^{-C^2\log(n_2)/4}, n_2^{-(C-1)\sqrt{n_2/n_1}}\big\},& i=2 \\
    2\max\{(i-1)n_i^{1-\log(n_i)}, n_i^{1-(k-1)(n_{i-1}/\sqrt{n_{i}})}\},& i =  3,\dots,d.
    \end{cases}
\end{equation}
}
\end{theorem}
\begin{proof}
See \Cref{append:general_hp}.
\end{proof}

\section{SDP relaxation for exact completion}\label{section:sdp-e}

In this section we introduce our SDP relaxation for exact rank-one tensor completion.
Our main result of the section is \Cref{thm:tightsdp2},
which shows that the relaxation solves the exact completion problem under the SR-propagation condition.
We also show that the SDP can be modified to handle instances that satisfy S-propagation.

This section is divided into four parts.
First, we introduce our relaxations and state \Cref{thm:tightsdp2}.
Then, we explain how to modify our SDP to handle more instances.
Next, we introduce some background on sum of squares. 
The last part is devoted to the proof of \Cref{thm:tightsdp2}.

\subsection{Relaxation and main theorem}

Our SDP relaxation relies on a characterization of rank one tensors in terms of quadratic equations.
It is known that a tensor $T$ is rank one if and only if
\begin{align*}
  T_{\alpha_1}T_{\alpha_2} = T_{\alpha_3}T_{\alpha_4},\qquad
  \forall \{\alpha_1,\alpha_2;\alpha_3,\alpha_4\}
  \in \mathcal{A}
\end{align*}

where $\mathcal{A} \subset \NN$ is the set of all squares.
For instance, in the case $\NN = [2]\times[2]\times[2]$ we have 12 quadratic equations,
given by the squares in \eqref{eq:squares12}.

Let $x = \vec(T) \in \RR^{N}$ be the vectorization of the tensor~$T$.
We denote $x_\alpha = T_\alpha$ for $\alpha \in \NN$.
Let the matrix $X = x x^T \in \SS^{N}$
with entries $X_{\alpha, \beta} = T_{\alpha} T_\beta$.
The tensor completion problem \eqref{eq:exactcompletion} can be equivalently phrased as
\begin{align*}
  \text{ find } \quad &x\in \RR^{N}, X \in \SS^N\\
  \quad\text{ s.t. }\quad
  &x_{\alpha} = \hat{T}_{\alpha}, \quad\forall \alpha \in \Omega,\\
  &X_{\alpha_1,\alpha_2} = X_{\alpha_3,\alpha_4}, \quad \forall \{\alpha_1,\alpha_2;\alpha_3,\alpha_4\} \in \mathcal{A},\\
  &X = x x^T
\end{align*}
Consider the following SDP relaxation:
\begin{equation} \label{eq:sdp2}
\tag{SDP-E}
\begin{aligned}
  \min_{x\in \RR^{N}, X \in \SS^N} \quad & \tr(X)\\
  \quad\text{ s.t. }\quad
  &x_{\alpha} = \hat{T}_{\alpha}, \quad\forall \alpha \in \Omega,\\
  &X_{\alpha,\alpha} = \hat{T}_{\alpha}^2, \quad\forall \alpha \in \Omega,\\
  &X_{\alpha_1,\alpha_2} = X_{\alpha_3,\alpha_4}, \quad \forall \{\alpha_1,\alpha_2;\alpha_3,\alpha_4\} \in \mathcal{A},\\
  &\boldsymbol{X}:=\begin{pmatrix} 1 & x^T \\ x & X \end{pmatrix} \succeq 0
\end{aligned}
\end{equation}
Note that condition 
$\left(\begin{smallmatrix} 1 & x^T \\ x & X \end{smallmatrix}\right) \succeq 0$
is equivalent to $X - x x^T\succeq 0$.
By minimizing the trace of $X$ we attempt to close the gap between $X$ and $x x^T$. 
\begin{remark}
     \eqref{eq:sdp2} is a simplification of the 2nd level of the Lasserre hierarchy.
     In particular, its matrix variable has size $N \times N$,
     instead of $O(n^d) 
     \times O(n^d)$, since it uses multilinear monomials.
     It also ignores the higher-order affine constraints, which have the form $X_{\alpha,\star} = \hat{T}_{\alpha} x$ for $\alpha \in \Omega$.
\end{remark}

From now on, we assume that the completion problem \eqref{eq:exactcompletion} has a unique solution,
and let $T_\beta$ for $\beta \in \NN$ be the completed values.
Let $\bar x \in \RR^N$ be the vector with entries $\bar x_\beta = T_\beta$.
Also consider the matrix
\(
    \boldsymbol{\bar X} :=
    \bigl(\begin{smallmatrix} 1 & \bar x^T\\ \bar x & \bar x\bar x^T \end{smallmatrix}\bigr).
\)
Notice that $\boldsymbol{\bar X}$ is feasible for \eqref{eq:sdp2}.

We say that the relaxation is \emph{tight} the SDP has a unique optimal solution that has rank-one, namely the matrix $\boldsymbol{\bar X}$.

A necessary condition for the SDP relaxation to be tight is that the completion problem \eqref{eq:exactcompletion} has a unique solution.
The main result of this section is that the SDP relaxation is tight when the S-propagation condition holds (which implies a unique completion). 

\begin{theorem}\label{thm:tightsdp2}
  Assume that $\Omega$ satisfies SR-propagation
  and that the values $(\hat{T}_\alpha)_{\alpha \in \Omega}$ are nonzero.
  Then the relaxation \eqref{eq:sdp2} is tight.
\end{theorem}

By \Cref{matrix_condition_equivalence} the propagation conditions are equivalent to unique recovery in the special case $d=2$.
Hence, we have the following corollary.

\begin{corollary}\label{thm:tightmatrix}
  In the case $d=2$
  (corresponding to matrix completion),
  the relaxation~\eqref{eq:sdp2} is tight if and only if the recovery problem has a unique solution.
\end{corollary}

A similar result to \Cref{thm:tightmatrix} was shown in \cite{cosse2021stable} for a different SDP relaxation,
involving a larger PSD matrix.

\subsection{Reweighted SDP with weaker assumptions}

The following example illustrates that there are instances where the completion problem has a unique solution,
but yet \eqref{eq:sdp2} is not {tight}.

\begin{example}\label{non_tight_example}
    Consider the mask $\Omega$ from \Cref{eg:Sholds},
    and let $T = (1,1,10)\otimes (1,1,1,10) \otimes (10,1,1)$ be the true tensor.
    An interior point solver returns an optimal solution of \eqref{eq:sdp2} of rank~19.
    Hence, the relaxation is not tight.
\end{example}

In order to deal with cases as the one in the example,
we consider variants of \eqref{eq:sdp2} obtained by modifying the trace objective by a weighted linear combination of the diagonal entries.
More precisely,
given a vector $w\in\RR^{N}$ of positive weights, consider:
\begin{equation} \label{eq:weighted_sdp}
\tag{SDP-$\mathrm{E}^w$}
\begin{aligned}
  \min_{x\in \RR^{N}, X \in \SS^N} \quad & \diag(w) \bullet X\\
  \quad\text{ s.t. }\quad
  &x_{\alpha} = \hat{T}_{\alpha}, \quad\forall \alpha \in \Omega,\\
  &X_{\alpha,\alpha} = \hat{T}_{\alpha}^2, \quad\forall \alpha \in \Omega,\\
  &X_{\alpha_1,\alpha_2} = X_{\alpha_3,\alpha_4}, \quad \forall \{\alpha_1,\alpha_2;\alpha_3,\alpha_4\} \in \mathcal{A},\\
  &\boldsymbol{X}:=\begin{pmatrix} 1 & x^T \\ x & X \end{pmatrix} \succeq 0
\end{aligned}
\end{equation}

The following theorems show that the use of weights allows to substitute the SR-propagation by the weaker GS-propagation conditions. 

\begin{theorem}\label{thm:tight_weighted_sdp}
  Assume that $\Omega$ satisfies GS-propagation
  and that the values $(\hat{T}_\alpha)_{\alpha \in \Omega}$ are nonzero.
  Then there exists a nonnegative weight vector $w$ such that the \eqref{eq:weighted_sdp} is tight and recovers the unique rank one tensor completion.
\end{theorem}
\begin{proof}
See \Cref{append:thm:tight_weighted_sdp}.
\end{proof}

Unfortunately, the above theorem does not indicate how to produce the vector~$w$.
If we use the S-propagation condition instead,
\Cref{alg:weights} from below generates a vector $w$ for which \eqref{eq:weighted_sdp} is tight.

\begin{theorem}\label{thm:tight_sdp_square}
  Assume that $\Omega$ satisfies S-propagation
  and that the values $(\hat{T}_\alpha)_{\alpha \in \Omega}$ are nonzero. Then there exists $t \in (0,1)$ such that
  for any $\vartheta \in (0,t)$
  \eqref{eq:weighted_sdp} is tight for
  the weight vector $w$ generated by \Cref{alg:weights}.
\end{theorem}
\begin{proof}
See \Cref{append:thm:tight_sdp_square}.
\end{proof}

The basic idea of \Cref{alg:weights} is to assign a smaller weight on the elements outside $\mathcal{P}_{SR}(\Omega)$, and keep decreasing the weights as the propagation iterates.

The following example illustrates that \eqref{eq:weighted_sdp} allows to solve the completion problem in an instance in which $\eqref{eq:sdp2}$ is not tight.

\begin{algorithm}[htb]
\caption{Weights generation}\label{alg:weights}
\begin{algorithmic}
\Require $\Omega \subset \NN$ satisfying S-propagation and a constant $\vartheta\in (0,1)$
\Ensure Weights $w_\beta$ for $\beta\in\NN$
\State $\mathcal{P}$ $\gets$  $\mathcal{P}_{SR}(\Omega)$ 
\State $w_\alpha \gets 1,\quad \forall \alpha \in \NN$ 
\While{$|\mathcal{P}| \neq N$}
\State $B \gets \{ \beta\in\NN\setminus \mathcal{P} : \{\alpha_i,\alpha_j;\alpha_k,\beta\} \text{~is a square} ~~~\text{for some}~~~ \alpha_i,\alpha_j,\alpha_k\in \mathcal{P} \}$
\State $w_\beta \gets \vartheta \cdot w_\beta, \quad\forall \beta\in B$
\State $\mathcal{P} \gets \mathcal{P} \cup B$
\EndWhile
\State \Return $w$
\end{algorithmic}
\end{algorithm}

\begin{example}
  Consider the completion problem in \Cref{non_tight_example}.
  \Cref{alg:weights} with $\epsilon \!=\! 0.01$ generates the weights
  $w_{3,4,1} \!=\! 0.01$ and
  $w_\beta \!=\! 1$ for all $\beta \!\neq\! (3,4,1)$.
  It can be checked that \eqref{eq:weighted_sdp} is tight with those weights.
\end{example}

\subsection{Background on Sums of Squares}

Here we recall some notions from Sum-of-Squares (SOS) which will be used in the proof of \Cref{thm:tightsdp2}. 
Let $\RR[x]$ be the ring of polynomials in variables~$x = (x_\alpha: \alpha \in \NN)$
and let $\RR[x]_2 := \{p\in\RR[x]:\deg(p)\leq 2\}$.
A polynomial $f(x) \in \RR[x]$ is SOS
if
\( f(x) = \sum_{i=1}^s (p_i(x))^2 \)
for some polynomials $p_i(x)\in \RR[x]$.
We denote
\[
  \Sigma_{N} := \{p\in \RR[x]: p ~~\text{is SOS}\},\qquad
  \Sigma_{N,2} = \Sigma_{N} \cap \RR[x]_2.
\]
It is known that a polynomial $f(x)\in \Sigma_{N,2}$ if and only if it can be written as
$$f(x) = [1;x]^T Q\ [1;x],$$
for some PSD matrix $Q \in \SS^{N+1}$.
We call $Q$ a \emph{Gram matrix} of~$f(x)$.

We introduce two truncated ideals in $\RR[x]_2$ related to the constraints in \eqref{eq:sdp2}.

First, let $I_2 \subset \RR[x]_2$ be generated by the constraints corresponding to the observed values:
\begin{align*}
  I_2:=\Bigl\{ \sum_{\alpha\in \Omega} &\Big(\nu_\alpha f_\alpha + \lambda_{\alpha}g_\alpha\Big):~
  \lambda_\alpha,\lambda'_\alpha\in \RR \Bigr\},\\
  &\text{ where }
  f_\alpha= x_\alpha - \hat{T}_{\alpha}, \quad
  g_\alpha= x_\alpha^2 - \hat{T}_{\alpha}^2.
\end{align*}
Second, let $J_2\subset \RR[x]_2$ be generated by constraints corresponding to the rank-one structure of the tensor:
\begin{align*}
  J_2 = \Bigl\{ \sum_{\{\alpha_i,\alpha_j;\alpha_k,\alpha_\ell\}\in\mathcal{A}}
    &\mu_{ijk\ell} (x_{\alpha_i} x_{\alpha_j} - x_{\alpha_k} x_{\alpha_\ell}):~
  \mu_{ijk\ell} \in \RR \Bigr\}
\end{align*}

In order to simplify the notation,
we define the following modulus operation.
Given a linear subspace $G \subset \RR[x]_2$ and polynomials $p,\tilde p \in \RR[x]_2$,
we denote
\begin{align*}
  p = \tilde p \bmod G
  &\quad\iff\quad p + g = \tilde p \text{ for some } g \in G\\
  p \in \Sigma_{N,2} \bmod G
  &\quad\iff\quad p + g \in \Sigma_{N,d} \text{ for some } g \in G
\end{align*}

The following proposition gives some simple properties of~$I_2$.

\begin{lemma}\label{lem:sos1}
     Let $h_{\alpha_1,\alpha_2} = x_{\alpha_1}x_{\alpha_2} - T_{\alpha_1} T_{\alpha_2}$, for $\alpha_1,\alpha_2\in \Omega$. Then
    \begin{equation}
        \pm h_{\alpha_1,\alpha_2} \in \Sigma_{N,2} \bmod I_2.
    \end{equation}
\end{lemma}
\begin{proof}
    Consider the following polynomials in $I_2$:
    $$
    q := g_{\alpha_1} + g_{\alpha_2} - 2T_{\alpha_2}f_{\alpha_1} -2T_{\alpha_1} f_{\alpha_2},\quad
    \hat q := g_{\alpha_1} + (T_{\alpha_1}/T_{\alpha_2})^2g_{\alpha_2}
    \quad \in\quad I_2. 
    $$ 
    The lemma follows from
    \[
    2h_{\alpha_1,\alpha_2} + q = (f_{\alpha_1} + f_{\alpha_2})^2,
    \quad
    -h_{\alpha_1,\alpha_2} + \hat q = (f_{\alpha_1} - (T_{\alpha_1}/T_{\alpha_2})f_{\alpha_2})^2.
    \qedhere
    \]
\end{proof}

\subsection{Proof of \Cref{thm:tightsdp2}}

Recall that $T_\beta$ for $\beta \in \NN$ are the entries of the unique completion.
The following lemma gives a sufficient condition for the SDP to be tight.

\begin{lemma}\label{lem:sdp-sufficient}
  Assume that
  \begin{equation}\label{eq:sdp-sufficient}
    \sum_{\beta\in\NN }(x_\beta^2 - T_\beta^2) = \sigma \bmod (I_2 + J_2),
  \end{equation}
  for some $\sigma \in \Sigma_{N,2}$ that admits a Gram matrix $Q \in \SS^{N+1}$ of rank~$N$.
  Then the relaxation \eqref{eq:sdp2} is tight.
\end{lemma}

In order prove the above lemma, we consider the dual problem of \eqref{eq:sdp2}:
\begin{equation} \label{eq:dual_sdp}
\tag{dSDP-E}
\begin{aligned}
  \max_{\lambda,\mu} \quad & \rho\\
  \quad\text{ s.t. }\quad
  &Q:= \begin{pmatrix} 0 & 0 \\ 0 & I_N \end{pmatrix} -\rho \begin{pmatrix} 1 & 0 \\ 0 & 0 \end{pmatrix} -\sum_{\alpha\in \Omega} \frac{\nu_\alpha}{2} \begin{pmatrix} -2\hat T_\alpha & e_\alpha^T \\ e_\alpha & 0 \end{pmatrix} \\
  &- \sum_{\alpha\in \Omega}\lambda_{\alpha}\begin{pmatrix} -\hat T_\alpha^2 & 0 \\ 0 & \diag(e_\alpha) \end{pmatrix} \\
  & -\sum_{\{\alpha_i,\alpha_j;\alpha_k,\alpha_\ell\}\in\mathcal{A}}\frac{\mu_{ijk\ell}}{2}\begin{pmatrix} 0 & 0 \\ 0 & E_{\alpha_i,\alpha_j} - E_{\alpha_k,\alpha_\ell} \end{pmatrix}\succeq 0,
\end{aligned}
\end{equation}
where $e_\alpha\in\RR^N$ is an indicator vector with 1 in the entry corresponding to $\alpha$ and 0's everywhere else, and $E_{\alpha,\beta} = e_\alpha e_\beta^T + e_\beta e_\alpha^T$.

\begin{proof}[Proof of \Cref{lem:sdp-sufficient}]
    Consider the definition of $Q$ in \eqref{eq:dual_sdp}.
    By multiplying by $[1;x]$ on the left and right we obtain:
    \begin{align*}
    [1; x]^T Q [1; x] =& \sum_{\beta \in \NN} x_\alpha^2 - \rho - \sum_{\alpha\in\Omega} (\nu_\alpha f_\alpha - \lambda_\alpha g_\alpha) \\
    &- \sum_{\{\alpha_i,\alpha_j;\alpha_k,\alpha_\ell\}\in \mathcal{A}} \mu_{ijk\ell} (x_{\alpha_i} x_{\alpha_j} - x_{\alpha_k} x_{\alpha_\ell})
    \end{align*}
    The above is very similar to \eqref{eq:sdp-sufficient},
    except that $\rho$ is replaced by $\sum_\beta T_\beta^2$.
    Hence, the assumption of the lemma implies that there exist some $\bar\nu_\alpha$,  $\bar\lambda_{\alpha}$, $\bar\mu_{ijk\ell}$, and $\bar\rho :=\sum_\beta T_\beta^2$ that are feasible for \eqref{eq:dual_sdp} and such that the corresponding matrix $\bar Q$ has rank~$N$.

    Recall that the matrix
    \(
        \boldsymbol{\bar X} :=
        \begin{pmatrix} 1 & \bar x^T\\ \bar x & \bar x\bar x^T \end{pmatrix}
    \)
    given by the unique completion is  feasible for the primal SDP.
    Note that $\bar X, \bar Q$ satisfy complementary slackness:
    \begin{equation}\label{eq: exact_dual_matrix}
        \bar Q \bullet \boldsymbol{\bar X}
        = [1;\bar x]^T \bar Q [1;\bar x]
        = \sigma(\bar x) = \sum_\beta (\bar x_\beta^2 - T_\beta^2) = 0.
    \end{equation}
    Hence, $\bar X, \bar Q$ are primal-dual optimal.
    
    Furthermore,
    any other primal optimal solution $\boldsymbol{X}^*$ satisfies
    $\bar Q \bullet \boldsymbol{X}^* = 0$.
    Since $\rank(\bar{Q}) = N$, then $\rank \boldsymbol{X}^* = 1$ and it must be a multiple of $\boldsymbol{\bar X}$.
    But since the first element is fixed to be $1$,
    then $\boldsymbol{X}^* = \boldsymbol{\bar X}$ is the unique primal solution, and the relaxation is tight.
\end{proof}

We next construct the SOS polynomial from \eqref{eq:sdp-sufficient}.

\begin{lemma}\label{lem:sdp2}
     If $\Omega$ satisfies SR-propagation, given any $\xi_\beta>0$ where $\beta \in \NN$, there exists  $\tilde p, \sigma \in \Sigma_{N,2}$ such that 
    \begin{equation}
       \sum_{\beta\in \NN}\xi_\beta(x_\beta^2 - T_\beta^2) = \sum_{\beta\in \NN}\eta_\beta\Big(x_\beta - \frac{T_\beta}{T_{\alpha_\beta}}x_{\alpha_\beta}\Big)^2 +\tilde p = \sigma\bmod (I_2 + J_2),
    \end{equation}
    where $\eta_\beta \in [\xi_\beta-\eta,\xi_\beta]$ is a positive constant, $\eta>0$ is a small constant
    and the SOS polynomial $\sigma$ has a Gram matrix $\bar Q$ of rank~$N$.  
\end{lemma}

\begin{proof}
    See \Cref{append:lem:sdp2}.
\end{proof}

By letting $\xi_\beta=1$ for any $\beta\in\NN$, \Cref{thm:tightsdp2} is a direct consequence of \Cref{lem:sdp-sufficient,lem:sdp2}.

\section{SDP relaxation for noisy completion}\label{section:sdp-n}

In this section we introduce our SDP relaxation for noisy rank-one tensor completion.
Our main result of the section is \Cref{thm:sdpp_stable}, 
which shows that the relaxation is able to construct all the entries of the rank-one tensor up to an error $O(\sqrt{\|\epsilon\|_\infty})$. This section is divided into two parts.
First, we introduce our relaxations and state \Cref{thm:sdpp_stable}.
The second part shows the proof of \Cref{thm:sdpp_stable}.

\subsection{Relaxation and main theorem}
In the noisy case the observed values are perturbed by a noise vector $\epsilon \in \RR^\Omega$, i.e.,
\[
T^\epsilon_\alpha = \hat{T}_\alpha + \epsilon_\alpha,\quad\forall \alpha\in \Omega,
\]
where $\{\hat{T}_\alpha\}_{\alpha\in \Omega}$ is the set of true values.
In this case we cannot rely on \eqref{eq:sdp2},
as it is typically infeasible.
Instead,
we turn some of the constraints into penalty terms in the objective,
obtaining the following SDP:
\begin{equation} \label{eq:sdp-p}
\tag{SDP-N}
\begin{aligned}
  \min_{x\in \RR^{N}, X \in \SS^N} \quad & \tr(X) + C\sum_{\alpha\in \Omega} \Big(X_{\alpha,\alpha} -2T^\epsilon_\alpha x_\alpha + ( T^\epsilon_\alpha)^2\Big) \\
  \quad\text{ s.t. }\quad
  \quad&X_{\alpha_1,\alpha_2} = X_{\alpha_3,\alpha_4}, \quad \forall \{\alpha_1,\alpha_2;\alpha_3,\alpha_4\} \in \mathcal{A},\\
  &\boldsymbol{X}:=\begin{pmatrix} 1 & x^T \\ x & X \end{pmatrix} \succeq 0
\end{aligned}
\end{equation}
where $C>0$ is a large constant. And its dual problem is, 
\begin{equation}\label{eq:sdp-p-d}
\tag{dSDP-N}
\begin{aligned}
  \max_{\rho_c^\epsilon, \mu} \quad & \rho^\epsilon_c\\
  \quad\text{ s.t. }\quad
  &Q^\epsilon_c:= \begin{pmatrix} 0 & 0 \\ 0 & I_N \end{pmatrix} -\rho^\epsilon_c \begin{pmatrix} 1 & 0 \\ 0 & 0 \end{pmatrix} +C\sum_{\alpha\in\Omega}A_\alpha\\
  & -\sum_{\{\alpha_i,\alpha_j;\alpha_k,\alpha_\ell\}\in\mathcal{A}}\frac{\mu^\epsilon_{ijk\ell}}{2}\begin{pmatrix} 0 & 0 \\ 0 & E_{\alpha_i,\alpha_j} - E_{\alpha_k,\alpha_\ell} \end{pmatrix}\succeq 0,
\end{aligned}
\end{equation}
where $A_\alpha = [-T^\epsilon_\alpha; e_\alpha][- T^\epsilon_\alpha; e_\alpha]^T$.

Recall that $\bar x$ is the vector given by the unique completion of the noiseless observations, and that
\(
    \boldsymbol{\bar X} :=
    \bigl(\begin{smallmatrix} 1 & \bar x^T\\ \bar x & \bar x\bar x^T \end{smallmatrix}\bigr).
\)
Our main result of the section, stated next,
shows that the distance of the minimizer of \eqref{eq:sdp-p} and $\boldsymbol{\bar X}$ is upper bounded by $O(\sqrt{\|\epsilon\|_\infty})$.
\begin{theorem}\label{thm:sdpp_stable}
     If the SR-propagation  is satisfied, then the optimal solution $\boldsymbol{X^\epsilon_c}$
    of \eqref{eq:sdp-p} obeys
     \begin{equation}\label{eq:noisy_stable}
       \|\boldsymbol{\bar X} - \boldsymbol{ X^\epsilon_c}\|_F \leq O\Big(\chi\big(\|\epsilon\|_\infty + \|\epsilon\|_\infty^2C + 1/C\big)\Big).
     \end{equation}
     where $\chi$ is a map: $\RR\rightarrow \RR,  x\rightarrow x+\sqrt{x}$, and $O(\cdot)$ is applied with respect to $\|\epsilon\|_\infty$ and $C$.
     In particular, for $C = 1/\|\epsilon\|_\infty$ the upper bound is $O(\sqrt{\|\epsilon\|_\infty})$.
\end{theorem}

\subsection{Proof of \Cref{thm:sdpp_stable}}

We start with some preliminary lemmas.
Firstly, we give a useful property satisfied by all  polynomials in~$I_2$.

\begin{lemma}\label{lem:poly_error}
    For any $\alpha\in\Omega$, let  $f^\epsilon_\alpha = x_\alpha - T^\epsilon_\alpha$.
    Given $C>0$ and a polynomial $p\in I_2$,
    there exists $\tilde p^\epsilon\in\Sigma_{N,2}$ such that
    \begin{equation}\label{eq:poly_error}
        p  + C\sum_{\alpha\in\Omega}(f^\epsilon_\alpha)^2 = \tilde p^\epsilon + O(1/C) + O(\|\epsilon\|_\infty).
    \end{equation}
    Moreover, plugging in $x_\alpha = \hat T_\alpha$ for $\alpha\in\Omega$ in the polynomial $\tilde p^\epsilon$ gives a scalar of magnitude $O\big(\|\epsilon\|_\infty + \|\epsilon\|_\infty^2C + 1/C\big)$.
\end{lemma}
\begin{proof}
   See \cref{append:lem:poly_error}
\end{proof}

The next lemma shows the existence of a feasible dual matrix such that its inner product with the optimal solution in \eqref{eq:sdp-p} is bounded.

\begin{lemma}\label{lem:dual_bound}
    Let $\boldsymbol{\bar X},\boldsymbol{X_c^\epsilon}$ be the optimal solutions of \eqref{eq:sdp2} and \eqref{eq:sdp-p}, and $\bar Q$ be the Gram matrix from \Cref{lem:sdp2}.
    There exists a dual feasible $\bar Q^\epsilon_c$ for \eqref{eq:sdp-p-d}
    such that
    \begin{align}\label{eq:dual_bound}
    \bar Q_c^\epsilon \bullet \boldsymbol{\bar X} \leq O(\|\epsilon\|_\infty+\|\epsilon\|_\infty^2C+1/C), \quad
    \bar Q_c^\epsilon\bullet \boldsymbol{X_c^\epsilon} \leq O(\|\epsilon\|_\infty + \|\epsilon\|_\infty^2C + 1/C).
    \end{align}
    In addition, we have that $\bar Q_c^\epsilon \succeq \bar Q$.
\end{lemma}

\begin{proof}
    See \Cref{append:lem:dual_bound}.
\end{proof}

We are ready to prove the stability of \eqref{eq:sdp-p} under the noise.

\begin{proof}[Proof of \Cref{thm:sdpp_stable}]
    Let
    \(
    \boldsymbol{X_c^\epsilon} = \sum_{i=1}^N (\lambda_c)_i u_iu_i^T
    \)
    be its the eigenvalue decomposition.
    From \Cref{lem:dual_bound}, we have
    \begin{equation}\label{eq1}
         \bar{Q} \bullet \boldsymbol{X_c^\epsilon}\leq \bar Q^\epsilon_c \bullet \boldsymbol{X_c^\epsilon} \leq O(\|\epsilon\|_\infty + \|\epsilon\|_\infty^2C + 1/C),
    \end{equation}
    Let $\phi: \RR^N \longmapsto\ker(\bar{Q})$ be the projection to $\ker(\bar{Q})$, and $\phi^\perp: \RR^n \longmapsto\ker(\bar{Q})^\perp$ be the projection to $\ker(\bar{Q})^\perp$.
    From \Cref{lem:sdp2}, it is true that $\rank(\bar Q) = N$. Let $\bar\tau>0$ be the the second smallest eigenvalue of $\bar Q$. Then it follows that
    
    \begin{equation}\label{eq:lower_bd_thm6}
    \begin{split}
         \bar Q\bullet \boldsymbol{X_c^\epsilon}=\sum_{i=1}^N(\lambda_c)_iu_i^T \bar{Q} u_i
         =& \sum_{i=1}^N (\lambda_c)_i\phi^\perp(u_i)^T \bar{Q} \phi^\perp(u_i)\\
       =&\sum_{i=1}^N (\lambda_c)_i\|\phi^\perp(u_i)\|^2\Bigg(\frac{ \phi^\perp(u_i)^T \bar{Q} \phi^\perp(u_i)}{\|\phi^\perp(u_i)\|^2}\Bigg)\\
       \geq & \bar \tau \sum_{i=1}^N(\lambda_c)_i\|\phi^\perp(u_i)\|^2
    \end{split}
    \end{equation}
    
    Combine \eqref{eq1} and \eqref{eq:lower_bd_thm6}, it can be shown that
    \begin{equation}\label{ineq:1}
    \sum_{i=1}^N(\lambda_c)_i\|\phi^\perp(u_i)\|_2^2 \leq \frac{\bar Q\bullet \boldsymbol{X^\epsilon_c}}{\bar \tau}\leq O(\|\epsilon\|_\infty + \|\epsilon\|_\infty^2C + 1/C).
    \end{equation}
    Denoting $u_0 = [1;\bar x]$, then we have
    \begin{equation}\label{eq:14}
        \begin{split}
         \boldsymbol{X_c^\epsilon}
         &= \sum_{i=1}^N (\lambda_c)_i (\phi(u_i) + \phi^\perp (u_i))(\phi(u_i) + \phi^\perp (u_i))^T
         = U_1 + U_2 + U_3
        \end{split}
    \end{equation}
    where 
    \begin{align}
        \label{eq:U1}
        U_1 :=& \big(\sum_{i=1}^N (\lambda_c)_i\big) \phi(u_i)\phi(u_i)^T,
        \\
        \label{eq:U2} U_2 :=& \sum_{i=1}^N (\lambda_c)_i (\phi(u_i) \phi^\perp (u_i)^T + \phi^\perp (u_i)\phi(u_i)^T),\\
        \label{eq:U3} U_3 :=&\sum_{i=1}^N (\lambda_c)_i \phi^\perp (u_i)\phi^\perp (u_i)^T.
    \end{align}
    Since $\boldsymbol{X_c^\epsilon}$ is the optimal solution of \eqref{eq:sdp-p}, then
    \begin{equation}\label{eq:opt_inequality}
    \tr\big(\boldsymbol{X_c^\epsilon}\big) + C\sum_{\alpha\in\Omega}(A_\alpha \bullet \boldsymbol{X_c^\epsilon}) \leq \tr(\boldsymbol{\bar X}) + C\sum_{\alpha\in\Omega} (A_\alpha \bullet \boldsymbol{\bar X}).    
    \end{equation}
    
    From the Proof of \Cref{lem:poly_error} (\Cref{append:lem:poly_error}), plugging $x_\alpha =\hat T_\alpha$ into $f^\epsilon_\alpha , \forall \alpha\in\Omega$, we have 
    \begin{equation}\label{eq:plug_in}
          C\sum_{\alpha\in\Omega} (A_\alpha \bullet \boldsymbol{\bar X}) =C\sum_{\alpha\in\Omega}(f^\epsilon_\alpha)^2 = O(\|\epsilon\|_\infty^2C)
    \end{equation}
    where the first equality is derived from the definition of $f^\epsilon_\alpha$ in \eqref{eq:poly_error} and $A_\alpha$ in \eqref{eq:sdp-p-d}, and the second equality follows from \eqref{eq:11}.
    Combining \eqref{eq:opt_inequality} and \eqref{eq:plug_in}, it follows
    \begin{equation}\label{eq: eigvals_bound}
        \begin{split}
         \sum_{i=1}^N (\lambda_c)_i = \tr(\boldsymbol{X_c^\epsilon}) \leq 
         \tr\big(\boldsymbol{X_c^\epsilon}\big) + C\sum_{\alpha\in\Omega}(A_\alpha \bullet \boldsymbol{X_c^\epsilon})
         \leq \tr(\boldsymbol{\bar X}) + O(\|\epsilon\|_\infty^2C).    
        \end{split}
    \end{equation}
    Then we derive that
    \begin{equation}\label{ineq:u2}
    \begin{split}
        \|U_2\|_F
        \leq& \sum_{i=1}^N 2(\lambda_c)_i \big\|\phi^\perp(u_i)\big\|_2\\
        \leq&2 \Big(\sum_{i=1}^N (\lambda_c)_i\Big)^{1/2}
        \Big(\sum_{i=1}^N (\lambda_c)_i\|\phi^\perp(u_i)\|_2^2\Big)^{1/2}\\
        \leq& \Big(\sqrt{\tr(\boldsymbol{\bar X})+O(\|\epsilon\|^2_\infty C)}\Big)O\Big(\sqrt{\|\epsilon\|_\infty + \|\epsilon\|_\infty^2C + 1/C}\Big)\\
        \leq & O\Big(\sqrt{\|\epsilon\|_\infty + \|\epsilon\|_\infty^2C + 1/C} + \big(\|\epsilon\|_\infty + \|\epsilon\|_\infty^2C + 1/C\big)\Big)
    \end{split}
    \end{equation}
    where the first inequality is from $\|\phi(u_i)\|_2\leq 1$ for $i=1,\dots,N$, the second inequality is derived from Cauchy–Schwarz inequality and $\|\phi(u_i)\|_2\leq 1$ for $i=1,\dots,N$, and the third inequality is from \eqref{ineq:1} and \eqref{eq: eigvals_bound}.
    
    Observe now that
    \begin{equation}\label{ineq:u3}
        \begin{split}
            \|U_3\|_F
            \leq &\sum_{i=1}^N (\lambda_c)_i \big\|\phi^\perp (u_i)\big\|_2^2
            \leq O\Big(\|\epsilon\|_\infty+\|\epsilon\|_\infty^2C + 1/C\Big).
        \end{split}
    \end{equation}
    Then from \eqref{ineq:u2} and \eqref{ineq:u3}, it follows that 
    \begin{equation}\label{ineq:2}
        |(U_2)_{1,1}|+|(U_3)_{1,1}|\leq O\Big(\chi\big(\|\epsilon\|_\infty+\|\epsilon\|_\infty^2C + 1/C\big)\Big)
    \end{equation}
    
    Since $\ker \bar Q$ is spanned by $u_0$, there exists $\tilde\lambda>0$ such that $U_1 = \tilde\lambda u_0u_0^T$.
    Recall that
    \( \boldsymbol{\bar X} = u_0u_0^T.\)
    Then we have
    \[
    (\boldsymbol{\bar X})_{1,1} =1= (\boldsymbol{X_c^\epsilon})_{1,1} = (U_1)_{1,1} + (U_2)_{1,1} + (U_3)_{1,1},
    \]
    where $(U_1)_{1,1} = \tilde\lambda(u_0u_0^T)_{1,1} = \tilde \lambda.$ Together with \eqref{ineq:2}, we have
    \begin{equation}\label{ineq:eig}
    1 - \tilde\lambda = (U_2)_{1,1} + (U_3)_{1,1}\leq O\Big(\chi\big(\|\epsilon\|_\infty+\|\epsilon\|_\infty^2C + 1/C\big)\Big).
    \end{equation}
    Using \eqref{ineq:eig}, \eqref{ineq:u2} and \eqref{ineq:u3}, we obtain
     \begin{equation}
         \begin{split}
             \|\boldsymbol{\bar X} - \boldsymbol{X_c^\epsilon}\|_F &= \|\boldsymbol{\bar X} - (U_1+U_2+U_3)\|_F\\
             &\leq\|(1 - \tilde\lambda) u_0 u_0^T\|_F 
             + \|U_2\|_F + \|U_3\|_F\\&\leq O\Big(\chi\big(\|\epsilon\|_\infty+\|\epsilon\|_\infty^2C + 1/C\big)\Big). \qedhere
         \end{split}
     \end{equation}
\end{proof}

\section{Experiments}

We proceed to discuss the computational performance of our SDP relaxations.
In \Cref{experiments:exact} and \Cref{experiments:noisy_small}, we focus on exact and noisy completion problems for small-scale tensors
using \texttt{Mosek} as the SDP solver,
and we compare our results with the sum-of-squares method in \cite{potechin2017exact}, iterative hard thresholding (IHT) via HOSVD, nuclear norm minimization via matrix unfolding and alternating minimization.
In \Cref{experiments:noisy_med} we approach medium-size tensors using a custom SDP solver based on library \texttt{NLopt}\cite{johnson2014nlopt},
and we compare our results with IHT and alternating minimization. In \Cref{experiments:application}, we propose a low-rank tensor completion approach based on \ref{eq:sdp-p} and then apply our method to image inpainting problem.
Our custom solver is presented in Appendix~\ref{experiments:custom_solver},
inspired by ideas from \cite{cosse2021stable}.
All the methods in this section are coded in \texttt{Julia}~1.8 and performed on a 2020 Macbook Pro with 8-core CPU and 16GB of memory.

\subsection{Exact completion: small-scale}\label{experiments:exact}
We generate uniformly random observation masks $\Omega$ of different cardinalities,
but we only keep masks which satisfy unique recovery.
First, we illustrate how the different propagation conditions behave when the number of observations is the the minimum that is required for unique recovery, namely $|\Omega| = n - d + 1$. We generate uniformly random observation masks $\Omega$ of different cardinalities,
but we only keep masks which satisfy unique recovery. \Cref{tab:exp_conditions} shows the number of masks $\Omega$ satisfying each of the propagation conditions for different choices of~$\NN$.
Notice that the A-propagation condition, used in \cite{cosse2021stable}, is almost never satisfied in this minimal setting.

\begin{table}[htb]
    \centering
    \caption{Percentage of $\Omega$ satisfying each of the propagation conditions, using the minimum number of observations.}
    \label{tab:exp_conditions}
    \begin{tabular}{ p{3cm}p{2.5cm}p{2.5cm}p{3.5cm} }
    \toprule
     & $[5]{\times}[5]{\times}[5]$ $|\Omega|=13$& $[3]{\times}[3]{\times}[3]{\times}[3]$ $|\Omega|=9$& $[2]{\times}[2]{\times}[2]{\times}[2]{\times}[2]$ $|\Omega|=6$\\
    \midrule
    Unique Recovery& 100\% &100\% &100\% \\
    GS-propagation & 100\% &100\% &100\% \\
    S-propagation  & 96\% &95\% &96\% \\
    SR-propagation  & 23\% &51\%&93\% \\
    A-propagation  & 0\% &0\% &1\% \\
    \bottomrule
    \end{tabular}
\end{table}

Next, we generate random rank-one tensors
\( T = u^{(1)}\otimes\dots\otimes u^{(d)}, \)
where each vector $u^{(i)}$ has i.i.d.~entries from the uniform distribution $U(0.1, 1)$.
We compare our \eqref{eq:sdp2} and its variant \eqref{eq:weighted_sdp}
against standard alternating minimization.
We also compare against the sum-of-squares method with degree~4 in \cite{potechin2017exact}, though this is only applicable for tensors of order~3.
For \eqref{eq:weighted_sdp} we use \Cref{alg:weights} to find the weights, using the parameter $\vartheta = 0.1$.

\Cref{fig:exact_rec} shows the percentage of experiments that are successfully solved (all entries are recovered exactly) by each of the methods with different number of observations.
Notice that \eqref{eq:sdp2} is consistently better than both alternating minimization and the method from \cite{potechin2017exact},
requiring significantly less observations to achieve exact recovery.
Regarding the weighted variant \eqref{eq:weighted_sdp},
it futher improves the recovery chances,
e.g., for $\NN = [5]\times[5]\times[5]$, $|\Omega|=13$ the percentage improves from 75 to 84,
and for $\NN = [3]\times[3]\times[3]\times [3]$, $|\Omega|=9$ it improves from 88 to 95.
We did not consider the weighted variant for $\NN = [2]\times[2]\times[2]\times[2]\times[2]$ since the recovery percentage of \eqref{eq:sdp2} is almost~100.

\begin{figure}[htb]
    \centering
    \subfigure[3-rd order: $5\times 5\times 5$]{
    \includegraphics[width=0.4\textwidth]{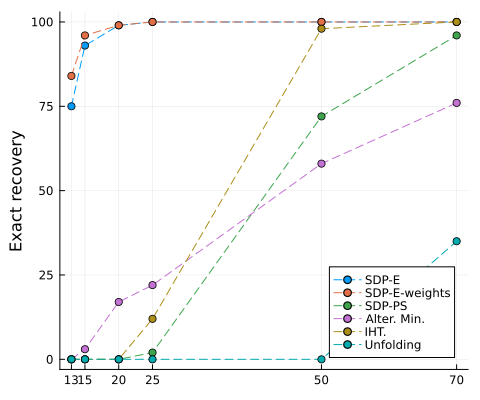}}
    \subfigure[4-th order: $3\times 3\times 3 \times 3$]{\includegraphics[width=0.4\textwidth]{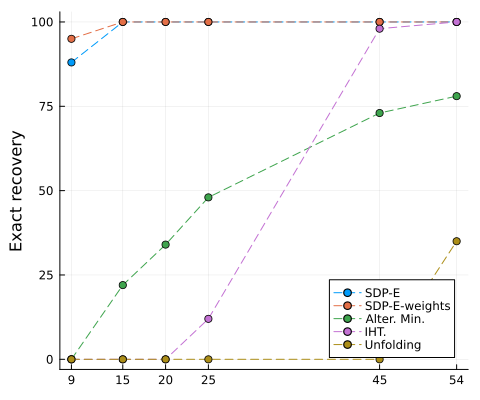}}
    \subfigure[5-th order: $2\times 2\times 2 \times 2 \times 2$]{
    \includegraphics[width=0.4\textwidth]{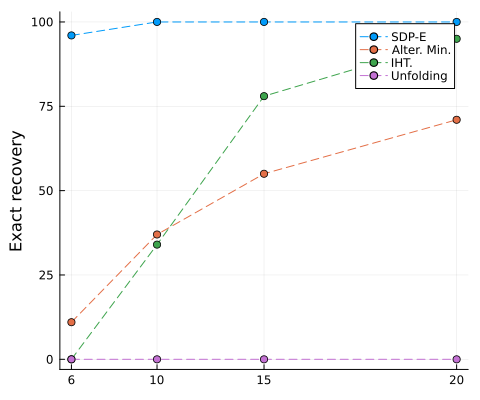}}
    \caption{Percentage of random completion problems that are recovered exactly by the methods: \eqref{eq:sdp2}, \eqref{eq:weighted_sdp}, SDP Potechin-Steurer~\cite{potechin2017exact}, alternating minimization.
    The $x$ axes is the number of observations used.}
    \label{fig:exact_rec}
\end{figure}

\subsection{Noisy completion: small-scale}\label{experiments:noisy_small}

We generate random rank-1 tensors, 
\( T = u^{(1)}\otimes\dots\otimes u^{(d)}, \)
where each vector $u^{(i)}$ has i.i.d.~entries from the uniform distribution $U(0.5, 1)$.
We then add i.i.d.\ random numbers $\epsilon_\alpha\sim\mathcal{N}(0,\delta^2)$ to each tensor entry, where $\delta$ is the noise level.
We also generate uniformly random observation masks $\Omega$ of different cardinalities.
We compare \ref{eq:sdp-p} against standard alternating minimization,
and also with with the degree~4 SDP of~\cite{potechin2017exact} for order 3 tensors.
We use a penalty constant $C = 100$ in all experiments.
The performance metric that we evaluate is the relative distance from the recovered rank-one tensors to the original tensor (before adding the noise).

\Cref{fig:small_dist} illustrates the relative distance achieved by each method for various numbers of observations and noise levels.
Notably, \eqref{eq:sdp-p} consistently yields better solutions than both the alternating minimization approach and the SDP method proposed in \cite{potechin2017exact}, across all examined settings.
In particular, the recovery of \eqref{eq:sdp-p} is orders of magnitude better than that of alternating minimization.
The benefits of \eqref{eq:sdp-p} are especially notable when the number of observations is limited.

\begin{figure}[htb]
    \centering
    \subfigure[$\NN = {[5]^3}, \delta=0.1$ ]{
    \includegraphics[width=0.3\textwidth]{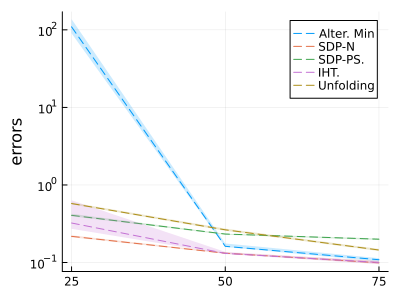}}
    \subfigure[$\NN = {[5]^3}, \delta=0.2$]{
    \includegraphics[width=0.3\textwidth]{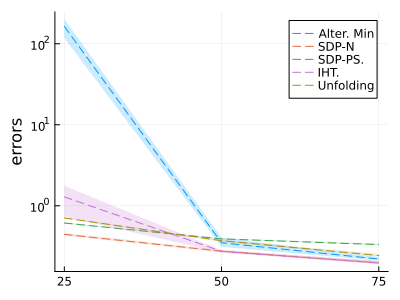}}
    \subfigure[$\NN = {[5]^3}, \delta=0.3$]{
    \includegraphics[width=0.3\textwidth]{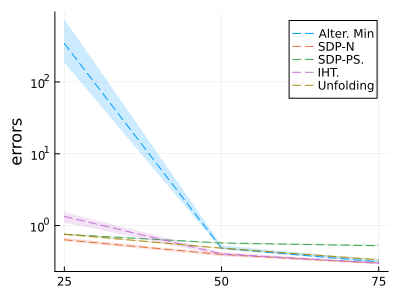}}

    \subfigure[$\NN = {[3]^4}, \delta=0.1$]{
    \includegraphics[width=0.3\textwidth]{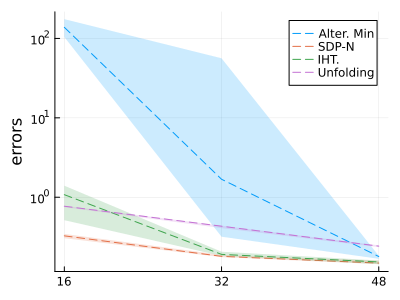}}
    \subfigure[$\NN = {[3]^4}, \delta=0.2$]{
    \includegraphics[width=0.3\textwidth]{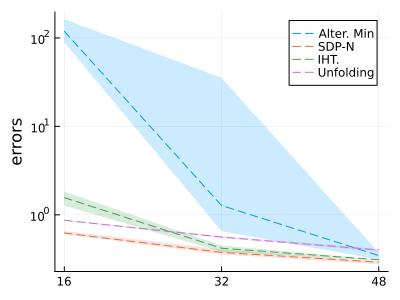}}
    \subfigure[$\NN = {[3]^4}, \delta=0.3$]{
    \includegraphics[width=0.3\textwidth]{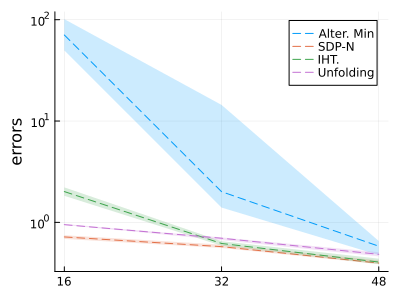}}

    \subfigure[$\NN = {[2]^5}, \delta=0.1$]{
    \includegraphics[width=0.3\textwidth]{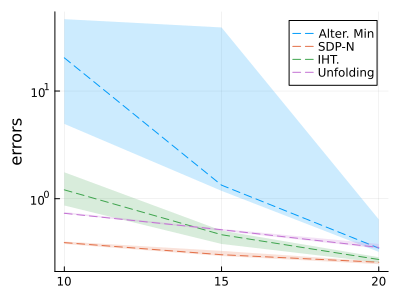}}
    \subfigure[$\NN = {[2]^5}, \delta=0.2$]{
    \includegraphics[width=0.3\textwidth]{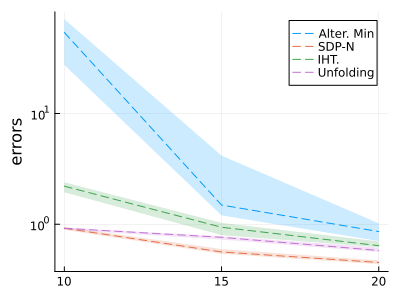}}
    \subfigure[$\NN = {[2]^5}, \delta=0.3$]{
    \includegraphics[width=0.3\textwidth]{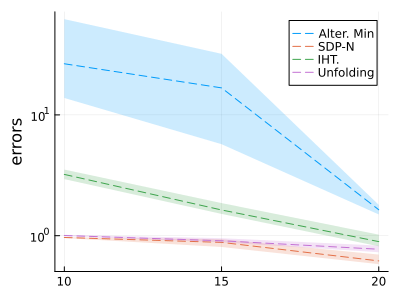}}

    \caption{Relative Distance (with 45\% - 55\% quantile) obtained by the methods: \eqref{eq:sdp-p}, SDP Potechin-Steurer~\cite{potechin2017exact}, alternating minimization.
    The $x$ axes is the number of observations used. }
    \label{fig:small_dist}
\end{figure}

\subsection{Noisy completion: medium-scale}\label{experiments:noisy_med} 

We generate the original tensor, the noisy tensor, and the observation mask in the same way as in \Cref{experiments:noisy_small},
but we now consider a larger tensors,
of dimension $30 \times 30 \times 30$.
Problem \eqref{eq:sdp-p} involves $7.3 \times 10^8$ scalar variables and $2.6 \times 10^8$ constraints,
so it is not possible to solve it with standard methods.
We use a custom SDP solver, which is presented in Appendix~\ref{experiments:custom_solver}.
Our custom solver is able to handle this large SDP,
though it is not as accurate as interior-point solvers such as Mosek.
We compare our method against standard alternating minimization.
We do not compare against the method in \cite{potechin2017exact}, since the SDP is also too large for standard solvers.
The metric we consider is the relative distance from the recovered rank-one tensor to the original tensor.

\Cref{fig:n30_dists} summarizes the results obtained.
With a limited number of observations, our method achieves substantially better recovery of the rank one tensor compared to alternating minimization,
and the difference can be of several orders of magnitude.
When the number of observations is much larger, alternating minimization exhibits a slight performance advantage over our method,
likely due to the fact that our custom SDP solver is not as accurate as interior-point methods.

\begin{figure}[htb]
    \centering
    \subfigure[$\NN = {[30]^3}, \delta=0.1$]{
    \includegraphics[width=0.3\textwidth]{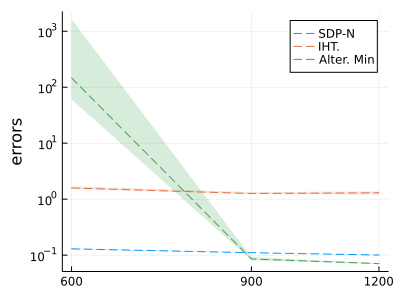}}
    \subfigure[$\NN = {[30]^3}, \delta=0.2$]{
   
    \includegraphics[width=0.3\textwidth]{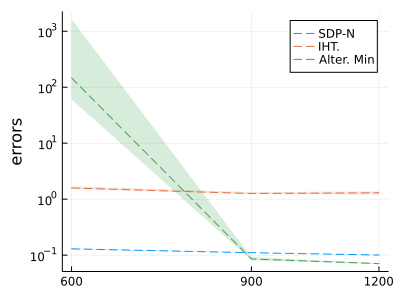}}
    \subfigure[$\NN = {[30]^3}, \delta=0.3$]{
    
    \includegraphics[width=0.3\textwidth]{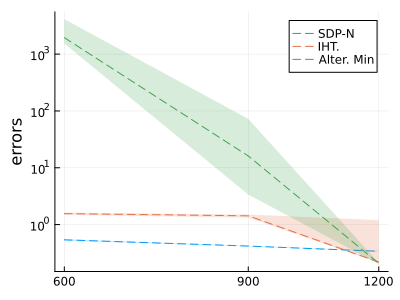}}

     \caption{Relative Distance (with 45\% - 55\% quantile) obtained by the methods: \eqref{eq:sdp-p}, SDP Potechin-Steurer~\cite{potechin2017exact}, alternating minimization.
    The $x$ axes is the number of observations used. }
    \label{fig:n30_dists}
\end{figure}

\subsection{Beyond rank-one tensor completion}\label{experiments:application}

We propose Algorithm~\ref{alg:low-rank} to extend our approach from rank-one to low-rank tensor completion. The method iteratively extracts rank-one components by solving a sequence of rank-one completion problems on the current residual; after each step, the residual is updated by subtracting the reconstructed component.

We applied the algorithm on an image inpainting problem. In \Cref{fig:img_uniform}, the missing entries are generated from uniform randomness. In \Cref{fig:img_rectangle}, a rectangular area is masked. We used \Cref{alg:low-rank} to recover the original image and the results are shown in \Cref{fig:img_rec_uniform} and \Cref{fig:img_rec_rectangle}.

\begin{algorithm}
\caption{Low-rank tensor completion via \ref{eq:sdp-p}}\label{alg:low-rank}
\begin{algorithmic}[1]
  \Require Observation index set $\Omega$, observed entries $\hat T_{\Omega}$, target rank $r$
  \Ensure An estimate $\tilde T$ of rank at most $r$
  \State $\tilde T \gets 0$ \Comment{current reconstruction}
  \State $R^{(0)}_{\Omega} \gets \hat T_{\Omega}$ \Comment{residual on observed set}
  \For{$k = 1,\dots,r$}
    \State $T_{k} \gets \mathrm{SDP\text{-}N}(\Omega, R^{(k-1)}_{\Omega})$
      \Comment{rank-one completion}
    \State $\tilde T \gets \tilde T + T_{k}$
    \State $R^{k}_{\Omega} \gets \hat T_{\Omega} - (\tilde T)_{\Omega}$ \Comment{update residual on $\Omega$}
  \EndFor
  \State \textbf{return} $\tilde T$
\end{algorithmic}
\end{algorithm}

\begin{figure}[htb]
    \centering
    \subfigure[Original image]{
    \includegraphics[width=0.3\textwidth]{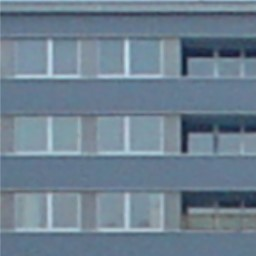}}
    \subfigure[Observed image]{\includegraphics[width=0.3\textwidth]{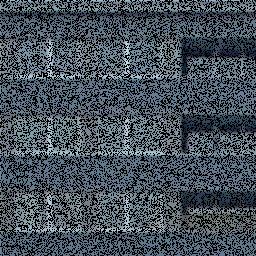}}
    \subfigure[Recovered image by SDP (rank-3 approximation with relative distance: 0.050)]{
    \includegraphics[width=0.3\textwidth]{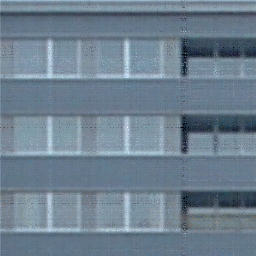}
    \label{fig:img_rec_uniform}
    }
    
    \caption{Application to image inpainting: $\NN = {256\times 256\times 3}$ }
    \label{fig:img_uniform}
\end{figure}

\begin{figure}[htb]\label{fig:img_rectangle}
    \centering
    \subfigure[Original image]{
    \includegraphics[width=0.3\textwidth]{figures/window2.png}}
    \subfigure[Observed image]{\includegraphics[width=0.3\textwidth]{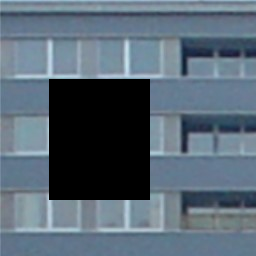}}
    \subfigure[Recovered image by SDP (rank-3 approximation with relative distance: 0.055)]{
    \includegraphics[width=0.3\textwidth]{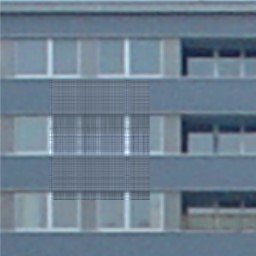}
    \label{fig:img_rec_rectangle}
     }
    
    \caption{Application to image inpainting: $\NN = {256\times 256\times 3}$ }
    \label{fig:img_rectangle}
\end{figure}

\bibliographystyle{abbrv}
\bibliography{refs}

\begin{appendices}
\section{Additional proofs}\label{Appendix}

\subsection{Proof of \Cref{thm:condition_equivalence}: A-propagation $\implies$ SR-propagation}\label{append:prop:AS-US}

\begin{proof}
    Let $\Omega$ satisfy A-propagation.
    We will show by induction that for each $s \in [d+1, n]$
    there exist $S_1\subseteq [n_1], \dots, S_d \subseteq [n_d]$ with $\sum_i |S_i| = s$,
    such that the grid $S_1\times\dots\times S_d$ is contained in $\mathcal{P}_{SR}(\Omega)$.
    In particular, $s = n$ leads to $\mathcal{P}_{SR}(\Omega) = \NN$,
    concluding the proof.
    For the base case, $s = d+1$,
    consider two tuples $\alpha, \alpha' \in \Omega$ that share all but one coordinate (which exist by A-propagation).
    Then $\{\alpha, \alpha'\}$ is a grid of dimensions $1,1,\dots,1,2$
    which is contained in $\mathcal{P}_{SR}(\Omega)$.
    It remains to show the inductive step.
    
    Assume the induction hypothesis holds for some $s< n$:
    there are $S_1,\dots,S_d$ with $\sum_i |S_i| = s$ such that
    the grid $G := S_1\times \dots\times S_d$
    is contained in $\mathcal{P}_{SR}(\Omega)$.
    Let $\Omega_1 = \Omega \cap G$ and $\Omega_2 = \Omega \setminus G$.
    By A-propagation,
    there exists $\alpha_1 \in \Omega_1$ and $\alpha_{2} \in \Omega_2$ that agree in all but one coordinate.
    So there is an index $j \in [d]$ such that
    $(\alpha_2)_i = (\alpha_1)_i \in S_i$ for $i \in [d] \setminus \{j\}$
    and $(\alpha_{2})_j \in [n_j] \setminus S_j$.
    Let $S_j' := S_j \cup \{(\alpha_2)_j\}$
    and consider the new grid
    the grid $G' := S_1\times \dots \times S_j' \times \dots\times S_d$.
    Let us show that $G' \subseteq \mathcal{P}_{SR}(\Omega)$.
    Let $\gamma = (i_1,\dots,(\alpha_2)_j,\dots,i_d) \in G'$.
    Denoting $\beta = (i_1,\dots,(\alpha_1)_j,\dots,i_d) \in G \subset \mathcal{P}_{SR}(\Omega)$,
    then $\{\alpha_{2}, \beta;\alpha_{1}, \gamma\}$ is a square,
    and hence $\gamma \in \mathcal{P}_{SR}(\Omega)$.
    It follows that $G' \subseteq \mathcal{P}_{SR}(\Omega)$, so the induction hypothesis holds for $s+1$ as well.
\end{proof}

\subsection{Proof of \Cref{general_hp}}\label{append:general_hp}

We will consider the special case $d=2$ before proceeding to the general case.

\begin{lemma}\label{lem_nonsym_mat}
    Let $d = 2$ and $3\leq n_1 \leq n_2$.
    Let $\Omega \subset \NN$ be a random subset,
    where each $(i_1, i_2)$ lies in $\Omega$, independently, with probability $p$, where
    \begin{align*}
         p\geq \frac{C}{\sqrt{n_1 n_2}} \log(n_2)
    \end{align*}
   for some constant $C\geq 2$. Then $\Omega$ satisfies unique recovery
   with probability at least $1-\varepsilon$, where  
   \[
   \varepsilon = 2\max \big\{n_2^{-C^2\log(n_2)/4}, n_2^{-(C-1)\sqrt{n_2/n_1}}\big\}
   \]
\end{lemma}

\begin{proof}
  We use the bipartite graph characterization of unique recovery from \Cref{matrix_condition_equivalence}.
  Hence, consider a bipartite graph with vertices $V = V_1 \cup V_2$,
  where each edge is realized with probability $p_1 p_2$. Consider cuts on $V_2$ into two sets of cardinalities $k$ and $n_2-k$.
  A cut on $V_2$ is \textit{disconnecting} if no vertex in $V_1$ is connected to both parts of the cut in~$V_2$.
  The log of the probability that there is a disconnecting cut of $V_2$ of cardinalities $k, n-k$ can be bounded by
  \begin{align*}
    \log\Big({n_2\choose k}(1-p_1^2p_2^2)^{n_1k(n_2-k)}\Big)
    \leq& \log\big( e^{-p_1^2p_2^2n_1k(n_2-k)}n_2^k\big)\\
    \leq& - C^2\log^2(n_2)k(1-k/n_2)+k\log(n_2)\\
    \leq&k\log(n_2)\Big(1-C^2\log(n_2)/2\Big),
  \end{align*}
  where $C>0$ is a constant. 
 
  Hence, the probability that there exists at least one disconnecting cut of $V_2$, of any cardinality, is at most
  \begin{align*}
    \sum_{k=1}^{n_2/2} e^{k\log(n_2)\big(1-C^2\log(n_2)/2\big)}
    =&\sum_{k=1}^{n_2/2}n_2^{k\big(1-C^2\log(n_2)/2\big)}\\
    \leq&(n_2^{1-C^2\log(n_2)/2})/(1-n_2^{1-C^2\log(n_2)/2}).
  \end{align*}
  Notice that the bipartite graph is disconnected only if we can find a disconnecting cut, or if some vertex of $V_1$ is not covered by any edge. When \(C\geq 2\) and \(n_2\geq 3\) we have \(n_2 \geq e^{4/C^2}\) which implies that
  \[
  1 - n_2^{-1}\leq n_2^{1-C^2\log(n_2)/2} \leq \exp(-\tfrac{C^2}{4}\log^2(n_2)).
  \]
  By the above analysis, the probability of the first event is at most
  \[
  \varepsilon_1 = n_2^{1-C^2\log(n_2)/2}/(1-n_2^{1-C^2\log(n_2)/2})\leq 2\exp\big(-\tfrac{C^2}{4}\log^2(n_2)\big).
  \]
  It remains to bound the probability that the edges do not cover the whole~$V_1$.
  For a fixed vertex in $V_1$, the probability that it is not covered by any edge is
  \begin{align*}
    (1-p_1p_2)^{n_2}
  \end{align*}
  Thus the probability that some vertex of $V_1$ is not covered is
  \begin{align*}
    \varepsilon_2 = n_1(1-p_1p_2)^{n_2}&\leq n_1\big(1-C\log(n_2)/\sqrt{n_1n_2}\big)^{n_2}
  \end{align*}
  The probability that the bipartite graph is disconnected (i.e., unique recovery fails) is at most $\varepsilon_1 + \varepsilon_2$.
  Since 
  \begin{equation*}
      \begin{split}
           n_1 \biggl(1- \frac{C\log n_2}{\sqrt{n_1n_2}}\biggr)^{n_2}
    =& n_1\exp\biggl(n_2\log\big(1 - \frac{C\log(n_2)}{\sqrt{n_1n_2}}\big)\biggr)\\
    \leq & n_1\exp\biggl(-n_2 \cdot \frac{C\log(n_2)}{\sqrt{n_1n_2}}\biggr)\\
    =&n_1 \exp\big(- C\log (n_2)\sqrt{n_2/n_1}\big)\\
    \leq &\exp\big( - (C-1)\log (n_2)\sqrt{n_2/n_1}\big)
      \end{split}
  \end{equation*}
  it follows that
  \[
  \varepsilon_1 + \varepsilon_2\leq 2n_2^{-t}
  \]
  where \(t:=\min \big\{\tfrac{C^2}{4}\log(n_2), (C-1)\sqrt{n_2/n_1}\big\}\geq 1\).
\end{proof}

We proceed to prove the theorem regarding an arbitrary~$d$.

\begin{proof}[Proof of \Cref{general_hp}]
    We proceed by induction on~$d$.
    The base case, $d=2$, follows from \Cref{lem_nonsym_mat}.
    Suppose the theorem holds in dimension $d-1$ and let us show it for dimension~$d$. With \(C\geq 2\prod_{i=3}^d (1 - n_i^{-\sqrt{n_i}})^{-1}\),
    we can find $p_1,\dots, p_d$ with
    \begin{equation*}
    p_1 \geq C/\sqrt{n_1},\qquad
    p_i \geq \log(n_i)/\sqrt{n_i}, \; i =  2,\dots,d
    \end{equation*}
    such that $p = p_1 p_2 \dots p_d$.
    Since the distribution on the tuples $(i_1,i_2,\dots, i_d)$ is uniform,
    we can equivalently interpret it as follows:
    for each $k\in [d]$ independently select a random subset $S_k$ of $[n_k]$, where each $i_k$ is with chosen with probability $p_k$,
    and then select $(i_1,i_2,\dots,i_d)$ if $i_k\in S_k$ for each~$k$. 
    
    Consider the projection $\varphi: [n_1]\times\dots\times[n_{d}]\rightarrow [n_1]\times\dots\times[n_{d-1}]$ with $\varphi(i_1,\dots,i_{d}) = (i_1,\dots,i_{d-1})$, and another projection $\psi: [n_1]\times \dots\times [n_{d}] \rightarrow [n_{d}]$ with $\psi(i_1,\dots,i_{d}) = i_{d}$. 
    The random set $\varphi(\Omega)$ contains each element in $[n_1]\times\dots\times[n_{d-1}]$ independently, with probability
    \begin{equation*}
        \begin{split}
             \Tilde{p} &\geq \Bigg (C\sqrt{\prod_{i=1}^{d-1}n_i^{-1}}\Bigg) \Bigg(\prod_{i=2}^{d-1}\log(n_i)\Bigg)(1-(1-p_d)^{n_d})\\
             &= \Bigg (C'\sqrt{\prod_{i=1}^{d-1}n_i^{-1}}\Bigg) \Bigg(\prod_{i=2}^{d-1}\log(n_i)\Bigg)
        \end{split}
    \end{equation*}
    where \(C' = C (1-(1-p_d)^{n_d})\). Since we have
    \begin{equation*}
        \begin{split}
            (1-p_d)^{n_d} =  \exp(n_d\log(1-p_d))
            \leq  \exp(-n_dp_d)
            =  n_d^{-\sqrt{n_d}}
        \end{split}
    \end{equation*}
    it follows that 
    \begin{align*}
        C'=C(1-(1-p_d)^{n_d})
        \geq 2\prod_{i=3}^{d-1} (1 - n_i^{-\sqrt{n_i}})^{-1}.
    \end{align*}
    Hence, $\varphi(\Omega)$ satisfies the assumption of the theorem in dimension~$d-1$.
    By the induction hypothesis, we have that $\varphi(\Omega)$ satisfies A-propagation with probability at least $1-\sum_{i=1}^{d-2}\varepsilon_i$, where each $\varepsilon_i$ is as in \eqref{bound whp}.
    From now on, we proceed conditionally on the event that $\varphi(\Omega)$ satisfies A-propagation.

    For a set $\Theta \subset [n_1] \times \dots \times [n_{d-1}]$ satisfying A-propagation 
    there is a strongly connected sequence
    $\beta_1,\beta_2,\dots, \beta_K \in \Theta$ of length $K = \sum_{i=1}^{d-1} n_i - d + 2$
    which cover $[n_1]\times\dots\times[n_{d-1}]$ coordinate-wise.
    For each possible $\Theta$ we choose one among all such sequences. 
    We call this the \emph{canonical} sequence of $\Theta$,
    and we denote as
    $\alpha_1(\Theta),\alpha_2(\Theta),\dots, \alpha_K(\Theta)$.
    From the definition of strong connectedness, for any $\alpha_k(\Theta)$, there exists $i_k < k$ such that $\alpha_{i_k}(\Theta)$ and $\alpha_k(\Theta)$ share all but one coordinate.
    Let $E(\Theta)$ consist of all such pairs $\{(i_k,k)\}_{2\leq k \leq K}$ of the canonical sequence.

    Let $\Omega$ such that $\varphi(\Omega)$ satisfies unique recovery.
    We introduce a sufficient condition for $\Omega$ to satisfy A-propagation in terms of the canonical sequence of $\varphi(\Omega)$.
    For $\ell=1,2,\dots,K$, consider the fibers of $\alpha_\ell(\varphi(\Omega))$ under $\varphi$:
    \[
        F_\ell(\Omega) = \{(i_1,\dots,i_{d})\in\Omega:\varphi(i_1,\dots,i_{d}) = \alpha_\ell(\varphi(\Omega))\}.
    \]
    A pair $(j,k)\in E(\varphi(\Omega))$ is called \textit{connected} if there exists
    $\beta_{j}\in F_{j}(\Omega)$, $\beta_k \in F_{k}(\Omega)$ such that
    $\beta_{j},\beta_k$ have $d-1$ indices in common
    (in particular,  $\psi(\beta_{j})=\psi(\beta_k)$).
    Otherwise, it is a \textit{disconnected} pair.
    A sufficient condition for $\Omega$ to satisfy A-propagation is that every pair in $E(\varphi(\Omega))$ is connected together with $\psi\big(\bigcup_{i=1}^K F_i(\Omega)\big) = [n_d]$.
    
    We proceed to upper bound the probability that $\Omega$ does not satisfy A-propagation,
    conditioned on $\varphi(\Omega)$ satisfying A-propagation.
    By the sufficient condition from above, this is upper bounded by
    \begin{align}\label{eq:two_probs_d}
        \Pr[\exists ij \in E(\varphi(\Omega)) \text{ s.t. } (i,j)\text{ disconnected } ] + 
        \Pr[ \psi\big(\cup_{i} F_i(\Omega)\big) \subsetneq [n_d] ]
    \end{align}
    We analyze the first probability. For a fixed pair $(i,j)\in E(\varphi(\Omega))$, the probability that this pair is disconnected is 
    \( (1 - p_d^2)^{n_d}. \)
    Since \(p_d\geq \log(n_d)/\sqrt{n_d},\)
    the probability can be bounded by
    \begin{equation*}
        \big(1-\log^2(n_d)/n_d\big)^{n_d}.
    \end{equation*} 
    The probability that there exists at least one disconnecting pair in $E(\varphi(\Omega))$ can be bounded by
    \begin{equation}\label{eq:general_whp_err_1}
        \begin{split}
             \hat \varepsilon_{d} & \leq K \big(1-\log^2(n_{d})/n_{d}\big)^{n_{d}}\\
        & = \exp\Big(\log((d-1)n_{d-1}) + n_d\log(1-\log^2(n_{d})/n_{d})\Big)\\
        & \leq \exp\Big(\log((d-1)n_{d-1}) - \log^2(n_{d})\Big)\\
        & \leq \exp\big(\log(d-1) + \log(n_{d})(1-\log(n_{d}))\big)\\
        & = (d-1) n_d^{1-\log(n_{d})}
        \end{split}
    \end{equation}
    
    We next bound the second probability in \eqref{eq:two_probs_d}. For a fixed index $\ell \in [n_d]$, the probability that $\ell \notin\psi\big(\cup_{i=1}^K F_i(\Omega)\big)$ is  
    \( (1-p_d)^{K}.\)
    The probability that $\psi\big(\cup_{i=1}^K F_i(\Omega)\big) \neq [n_{d}]$ is bounded by
    \begin{equation}
        \begin{split}\label{eq:general_whp_err_2}
        \Tilde\varepsilon_{d}&\leq n_{d}\big(1-\log(n_{d})/\sqrt{n_{d}}\big)^{K}\\
        &\leq \exp\Big(\log(n_{d}) + (d-1)n_{d-1}\log\big(1-\log(n_{d})/\sqrt{n_{d}}\big)\Big)\\
        &\leq \exp\Big(\log(n_{d}) - (d-1)n_{d-1}\big(\log(n_{d})/\sqrt{n_{d}}\big)\Big)\\
        &\leq \exp\big(\log(n_{d})(1-(d-1)n_{d-1}/\sqrt{n_{d}})\big)\\
        & = n_d^{1 - (d-1)(n_{d-1}/\sqrt{n_d})}
        \end{split}
    \end{equation}
    Hence, the quantity in \eqref{eq:two_probs_d} is bounded by
    $1 - \hat\varepsilon_{d} -\Tilde\varepsilon_{d}
    $.
    
    We are ready to bound the probability that $\Omega$ satisfies A-propagation.
    The probability that $\varphi(\Omega)$ satisfies A-propagation is $1-\sum_{i=2}^{d-1} \varepsilon_i$.
    The probability probability that $\Omega$ satisfies A-propagation conditioned on $\varphi(\Omega)$ satisfying it is
    $1 - \hat\varepsilon_d - \tilde\varepsilon_d$.
    Hence, $\Omega$ satisfies A-propagation with probability at least 
     \[
     (1-\sum_{i=2}^{d-1}\varepsilon_i)(1 - \hat\varepsilon_{d} -\Tilde\varepsilon_{d}) \geq 1 - \sum_{i=2}^{d} \varepsilon_i. 
     \]
     where $\varepsilon_d$ is as in \eqref{bound whp}.
    \end{proof}

\subsection{Proof of \Cref{lem:sdp2}}\label{append:lem:sdp2}
Throughout this section, we use the following notations.
For $\alpha,\beta\in\NN$, let $t_{\alpha,\beta} = T_\alpha/T_\beta$.
Observe that
\begin{equation}\label{eq:1}
q_{\alpha,\beta} := g_\alpha + (t_{\alpha,\beta})^2g_\beta - 2t_{\alpha,\beta}h_{\alpha,\beta} = (f_\alpha - t_{\alpha,\beta}f_\beta)^2.
\end{equation}
And for $\alpha,\beta,\gamma\in\NN$ and $\delta,\zeta\in\RR$, we have
\begin{equation}\label{eq:2}
\begin{split}
\bar{q}(\alpha,\beta,\gamma;\delta,\zeta):=&\delta^2g_\alpha + (\zeta/\delta)^2g_\beta + (\delta t_{\alpha,\gamma} + \zeta t_{\beta,\gamma}/\delta)^2g_\gamma + 2\zeta h_{\alpha,\beta} \\
&-2(\delta t_{\alpha,\gamma} + \zeta t_{\beta,\gamma}/\delta)(\delta h_{\alpha,\gamma} + (\zeta/\delta)h_{\beta,\gamma})\\
=&\Big(\delta f_\alpha + (\zeta/\delta)f_\beta - (\delta t_{\alpha,\gamma} + \zeta t_{\beta,\gamma}/\delta)f_\gamma\Big)^2.
\end{split}
\end{equation}

Before proceeding to the proof of \Cref{lem:sdp2}. We start with the following definition.

\begin{definition}
     Given \(\Omega\) which satisfies RS-propagation, for any \(\beta \in \NN\setminus\Omega\), a sequence $s(\beta) = (\beta^{(0)}, \beta^{(1)}, \dots, \beta^{(m_\beta)})$ is a \textit{SR-propagation sequence of} \(\beta\) if it satisfies
     \begin{enumerate}[label = (\roman*)]
         \item $\beta^{(0)} = \beta$, $\beta^{(m_\beta)} \in \Omega$ where \(m_\beta\in\mathbb{Z}\) is a constant which depends on \(\beta\).
         \item For \(i=0,\dots,m_\beta-1\), there exists \(\alpha_i,\tilde\alpha_{i+1}\in\Omega\) such that 
         \[
         \{\alpha_i,\beta^{(i)};\tilde\alpha_{i+1},\beta^{(i+1)}\}
         \]
         is a square.
     \end{enumerate}
\end{definition}

\begin{lemma}\label{lem:propagation_sequence_induction}
   Given a SR-propagation sequence $s(\beta) = (\beta^{(0)}, \beta^{(1)}, \dots, \beta^{(m_\beta)})$, for any \(\ell = 1,\dots,m_\beta\), there exists constants \(\{\delta_0,\dots,\delta_{\ell-1}, c_\ell\}\) and a polynomial \(\tilde p_\ell \in \Sigma_{N,2}\) such that 
    \begin{equation}\label{eq:prop_induction}
    \begin{split}
    \sum_{i=0}^{\ell-1} \delta_{i}^2 g_{\beta^{(i)}} - 2c_{\ell}h_{\beta^{(\ell)}, \tilde \alpha_\ell}
    = \delta_{0}(f_{\beta^{(0)}} -  t_{\beta^{(0)},\alpha_0}f_{\alpha_0})^2 + \tilde p_\ell \bmod (I_2 + J_2)\\
    \end{split}
\end{equation}
\end{lemma}

\begin{proof}[Proof of \Cref{lem:propagation_sequence}]
We will prove this lemma by induction. Consider the base case, $\ell=1$.
There exists $\alpha_0,\tilde \alpha_1 \in \Omega$ such that
$\{\alpha_0,\beta^{(0)};\tilde\alpha_1,\beta^{(1)}\}$ is a square.
Let constant $c_1 := \delta^2_{0}t_{\beta^{(0)},\alpha_0}$.
From \eqref{eq:1}, we obtain
\begin{equation}\label{eq:base}
\begin{split}
    &\delta^2_{0}g_{\beta^{(0)}} -2c_1h_{\beta^{(1)},\tilde\alpha_1}\\
    =& \delta^2_{0}\Big(g_{\beta^{(0)}} + (t_{\beta^{(0)},\alpha_0})^2g_{\alpha_0} - 2t_{\beta^{(0)},\alpha_0}h_{\beta^{(1)},\tilde\alpha_1}\Big)\bmod (I_2 + J_2)\\
    =& \delta^2_{0}\Big(g_{\beta^{(0)}} + (t_{\beta^{(0)},\alpha_0})^2g_{\alpha_0} - 2t_{\beta^{(0)},\alpha_0}h_{\beta^{(0)},\alpha_0}\Big)\bmod (I_2 + J_2)\\
    =& \delta^2_0(f_{\beta^{(0)}} - t_{\beta^{(0)},\alpha_0}f_{\alpha_0})^2 \bmod (I_2 + J_2)
\end{split}
\end{equation}
where the second equality holds because $\{\alpha_0,\beta^{(0)};\tilde\alpha_1,\beta^{(1)}\}$ is a square which implies that $h_{\beta^{(0)},\alpha_0} = h_{\beta^{(1)},\tilde\alpha_1} \bmod J_2$.
We have shown that \eqref{eq:prop_induction} holds for the base case.

Suppose the induction hypothesis \eqref{eq:prop_induction} holds for some $\ell < m$. There exists $\alpha_\ell,\tilde\alpha_{\ell+1}\in\Omega$ such that $\{\alpha_\ell,\beta^{(\ell)};\tilde\alpha_{\ell+1},\beta^{(\ell+1)}\}$ is a square. Let $c_{\ell+1} = \delta_{\ell}^2 t_{\beta^{(\ell)},\alpha_\ell} + c_\ell  t_{\tilde \alpha_\ell,\alpha_\ell}$.
By the definition of $\bar q$ and $c_{\ell+1}$,
we have that
\begin{equation}
\begin{split}
     & \sum_{i=0}^{\ell} \delta_{i}^2 g_{\beta^{(i)}} - 2c_{\ell+1}h_{\beta^{(\ell+1)}, \tilde \alpha_{\ell+1}}\\
     =&\Big(\delta_{\ell}^2g_{\beta^{(i)}} - 2c_{\ell+1}h_{\beta^{(\ell)}, \alpha_{\ell}} + 2c_{\ell}h_{\beta^{(\ell)}, \tilde \alpha_\ell} - 2\Big(c_\ell/\delta_{\ell}^2\Big)c_{\ell+1}h_{\tilde\alpha_{\ell}, \alpha_{\ell}}\Big)\\
     &+ \Big(\sum_{i=0}^{\ell-1}\delta_{i}^2 g_{\beta^{(i)}} - 2c_{\ell}h_{\beta^{(\ell)}, \tilde \alpha_\ell}\Big) +2\Big(c_\ell/\delta_{\ell}^2\Big)c_{\ell+1}h_{\tilde\alpha_{\ell}, \alpha_{\ell}} \bmod (I_2 {+} J_2)\\
     =&\bar q\big(\beta^{(\ell)},\tilde \alpha_\ell, \alpha_\ell;\delta_\ell, c_\ell\big)  + \Big(\sum_{i=0}^{\ell-1} \delta_{i}^2 g_{\beta^{(i)}} - 2c_{\ell}h_{\beta^{(\ell)}, \tilde \alpha_\ell}\Big) \\
     &+2\Big(c_\ell/\delta_{\ell}^2\Big)c_{\ell+1}h_{\tilde\alpha_{\ell}, \alpha_{\ell}}\bmod (I_2 {+} J_2)\\
\end{split}
\end{equation}
where the first equality holds because $\{\alpha_\ell,\beta^{(\ell)};\tilde\alpha_{\ell+1},\beta^{(\ell+1)}\}$ is a square, and then $h_{\beta^{(\ell+1)}, \tilde \alpha_{\ell+1}} = h_{\beta^{(\ell)},  \alpha_{\ell}} \bmod J_2$. The second equality follows from that $g_{\alpha_\ell}=g_{\tilde\alpha_\ell}= 0 \bmod I_2$.

From \Cref{lem:sos1}, we have that \(\pm h_{\tilde\alpha_\ell,\alpha_\ell}\in\Sigma_{N,2}\bmod I_2\). Together with \eqref{eq:2} and \eqref{eq:prop_induction}, we conclude that
\begin{equation}
\begin{split}
     \sum_{i=0}^{\ell}&\delta_{i}^2 g_{\beta^{(i)}} - 2c_{\ell+1}h_{\beta^{(\ell+1)}, \tilde \alpha_{\ell+1}(\beta)}\\
     =& \delta_{0}^2(f_{\beta^{(0)}} -  t_{\beta^{(0)},\alpha_0}f_{\alpha_0})^2 + \tilde p_{\ell+1} \bmod (I_2 + J_2)
\end{split}
\end{equation}
for some $\tilde p_{\ell+1}\in\Sigma_{N,2}$.
So \eqref{eq:prop_induction} holds for $\ell+1$,
completing the induction.
\end{proof}

\begin{lemma}\label{lem:propagation_sequence}
 Given a SR-propagation sequence $s(\beta) = (\beta^{(0)}, \beta^{(1)}, \dots, \beta^{(m_\beta)})$, there exists constants \(\{\delta_{\beta^{(i)}}\}_{i=0,\dots,m_\beta}\) and \(\tilde p_\beta \in \Sigma_{N,2}\) such that
\begin{equation}\label{eq:indction_end}
   \sum_{i=0}^{m_\beta} \delta_{\beta^{(i)}}^2 g_{\beta^{(i)}} = \delta_{\beta^{(0)}}^2(f_{\beta} -  t_{\beta,\alpha_0}f_{\alpha_0})^2+ \tilde p_{\beta} \bmod (I_2 + J_2),
\end{equation}
\end{lemma}
\begin{proof}
    From \Cref{lem:propagation_sequence_induction}, consider \eqref{eq:prop_induction} for $\ell = m_\beta$. Since $\beta^{(m_\beta)}\in\Omega$, then
\begin{equation}\label{eq:induction_0}
    \pm h_{\beta^{(m_\beta)},\tilde \alpha_{m_\beta}} \in \Sigma_{N,2} \bmod I_2
\end{equation}
Note that \(\beta^{(0)}=\beta\). Combining \eqref{eq:prop_induction} and \eqref{eq:induction_0} leads to \eqref{eq:indction_end}.
\end{proof}

Then we are ready to prove \Cref{lem:sdp2}.
\begin{proof}[Proof of  \Cref{lem:sdp2}]

Given \(\beta\in\NN\setminus\Omega\), there exists a SR-propagation sequence \(s(\beta) = (\beta^{(0)},\dots,\beta^{(m_\beta)})\). From \Cref{lem:propagation_sequence}, there exists constants \(\{\delta_{\beta^{(i)}}\big(s(\beta)\big)\}_{i=0,\dots,m_\beta}\) and \(\tilde p_\beta \in \Sigma_{N,2}\) such that \eqref{eq:indction_end} holds.
Note that $\beta$ may also appear on the sequences of other elements in $\NN\setminus\Omega$, that is, there exists $\hat \beta$ with SR-propagation sequence $s(\hat \beta) = (\hat\beta^{(0)},\dots,\hat\beta^{(m_{\hat\beta})})$ such that $\beta = \hat\beta^{(r)}$ for some $r\in [m_{\hat\beta}]$. Let
$\mathcal{Y}_{\beta}:=\big\{\hat\beta\in\NN\setminus\Omega: \exists~r\in [m_{\hat\beta}] \text{~s.t.~} \beta = \hat\beta^{(r)}\big\}$ be the set of indices whose SR-propagation sequences contain $\beta$. Note that \(\beta\) appears in the SR-propagation sequence \(\{s(\hat \beta): \hat\beta \in\mathcal{Y}_\beta\}\) and \(s(\beta)\).

Additionally, let \(\delta^2_{\beta}\big(s(\hat \beta)\big)\) denote the coefficient of $g_{\beta}$ within the SR-propagation sequences of $s(\hat{\beta})$, analogous to \eqref{eq:indction_end}.
Let $\varrho_{\beta}:= \sum_{\hat \beta\in\mathcal{Y}_{\beta}}\delta^2_{\beta}\big(s(\hat \beta)\big)$. Notice that each $\delta^2_{\beta}\big(s(\hat \beta)\big)$ can be chosen arbitrarily small. Hence, $\varrho_{\beta}$ can also be arbitrarily small. On the SR-propagation sequence of \(\beta\), we have \(\delta_{\beta^{(0)}}(s(\beta))=\delta_{\beta}(s(\beta))\). 
 
Let \(\varrho_{\beta}\) and \(\delta_{\beta}^2(s(\beta))\) be chosen to satisfy $\varrho_{\beta} + \delta^2_{\beta^{(0)}}(s(\beta))=\varrho_{\beta} + \delta_{\beta}^2(s(\beta))=\xi_
\beta$. (If $\mathcal{Y}_{\beta} = \emptyset$, then $\varrho_{\beta} = 0$). By applying \eqref{eq:indction_end} for all $\beta\in\NN\setminus\Omega$ and defining $\eta_\beta :=\xi_\beta - \varrho_{\beta}= \delta^2_{\beta}(s(\beta))$, it follows that
\begin{equation}\label{eq:combine_all}
    \sum_{\beta\in\NN\setminus\Omega} \xi_\beta g_{\beta} = \sum_{\beta\in\NN} \xi_\beta g_{\beta} = \sum_{\beta\in \NN}\eta_{\beta}\Big(f_{\beta} - t_{\beta,\alpha_{\beta}}f_{\alpha_{\beta}}\Big)^2 + \tilde p = \sigma \bmod (I_2 + J_2)
\end{equation}
where $\alpha_{\beta}\in \Omega$ is related to ${\beta}$, and each $\eta_{\beta}$ is close to $\xi_\beta$ arbitrarily. The first equality in \eqref{eq:combine_all} follows from the definition of $I_2$, and the second equality is a consequence of \eqref{eq:indction_end}.

To demonstrate that the rank of associated Gram matrix of $\sigma$ is of $N$, we will show the associated Gram matrix of 
\begin{equation}\label{eq:sos_high_rank}
    \sum_{\beta\in \NN}\eta_\beta\Big(f_\beta - t_{\beta,\alpha_\beta}f_{\alpha_\beta}\Big)^2 =\sum_{\beta\in \NN}\eta_{\beta}\Big(x_\beta - \frac{T_\beta}{T_{\alpha_\beta}}x_{\alpha_\beta}\Big)^2,
\end{equation}
is of rank of $N$.
Recall the definition of $e_\beta$ in \eqref{eq:dual_sdp} and let $r_\beta = [0;e_\beta - (T_\beta/T_{\alpha_\beta})e_{\alpha_\beta}]\in\RR^{N+1}$, it then follows that the Gram matrix of \eqref{eq:sos_high_rank}  is
\[
\sum_{\beta\in \NN}\eta_{\beta}  r_\beta r_\beta^T,
\]
Since $\{r_\beta\}_{\beta\in \NN}$ are linearly independent, it can be shown that
\[
\rank\Big(\sum_{\beta\in \NN}\eta_{\beta} r_\beta r_\beta^T\Big) = N.
\] 
Hence, the rank of Gram matrix of $\sigma$ is greater or equal to $N$. Clearly, $[1;\bar x]$ is in the kernel of this Gram matrix. Thus, its rank must be~$N$.
\end{proof}
    
\subsection{Proof of \Cref{thm:tight_weighted_sdp}}\label{append:thm:tight_weighted_sdp}
To prove this theorem, we follow a strategy similar to that used in \Cref{thm:tightsdp2}. In the light of \Cref{lem:tight_weighted_sdp}, a SOS polynomial can be constructed. We begin with the dual of \eqref{eq:weighted_sdp},

\begin{equation} \label{eq:dual_weighted_sdp}
\tag{dSDP-$\mathrm{E}^w$}
\begin{aligned}
  \max_{\lambda,\mu} \quad & \rho\\
  \quad\text{ s.t. }\quad
  &Q:= \begin{pmatrix} 0 & 0 \\ 0 & \diag(w)\end{pmatrix} -\rho \begin{pmatrix} 1 & 0 \\ 0 & 0 \end{pmatrix} -\sum_{\alpha\in \Omega} \frac{\nu_\alpha}{2} \begin{pmatrix} -2\hat T_\alpha & e_\alpha^T \\ e_\alpha & 0 \end{pmatrix} \\
  &- \sum_{\alpha\in \Omega}\lambda_{\alpha}\begin{pmatrix} -\hat T_\alpha^2 & 0 \\ 0 & \diag(e_\alpha) \end{pmatrix} \\
  & -\sum_{\{\alpha_i,\alpha_j;\alpha_k,\alpha_\ell\}\in\mathcal{A}}\frac{\mu_{ijk\ell}}{2}\Bigg(\begin{pmatrix} 0 & 0 \\ 0 & E_{\alpha_i,\alpha_j}\end{pmatrix}-\begin{pmatrix} 0 & 0 \\ 0 & E_{\alpha_k,\alpha_\ell} \end{pmatrix}\Bigg)\succeq 0.
\end{aligned}
\end{equation}

\begin{lemma}\label{lem:tight_weighted_sdp}
    Assume that $\Omega$ satisfies GS-propagation.
    Then for any $ \beta\in \NN\setminus \Omega$ and any constant $c_\beta$ there exists a set of indices $S_\beta\subseteq \NN$ and a set of nonnegative weights $\{w_\gamma^{(\beta)}\}_{\gamma\in S_\beta}$ such that 
    \begin{equation}\label{eq:prop_weighted}
         (x_\beta^2 - T_{\beta}^2) + \sum_{\gamma\in S_\beta} w^{(\beta)}_\gamma (x_{\gamma}^2 - T_\gamma^2) = -2c_{\beta}(x_\beta - T_{\beta}) + \tilde p_\beta \bmod (I_2 + J_2), 
    \end{equation}
    for some $\tilde p_\beta \in \Sigma_{N,2}.$
    If, in addition, $\Omega$ satisfies S-propagation,
    then
    $$ S_\beta \subseteq \cup_{r=0}^{\ell} \tilde B_r
    \text{ for all }\beta \in \tilde B_{\ell+1},$$
    where, $\tilde B_0 = \mathcal{P}_{SR}(\Omega)$ and
    \begin{align*}
     \tilde B_{\ell} = \{\beta\in\NN\setminus\Omega: \exists~ \alpha_i,\alpha_j,\alpha_k\in\cup_{r=0}^{\ell-1} \tilde B_u\text{~s.t.~}&\\
     \{\beta,\alpha_i;\alpha_j,\alpha_k\}\text{~is a square}&\}
    \end{align*} 
    \end{lemma}
    \begin{proof}
    Denote $B_0:=\Omega$, and
    $$B_\ell:=\{\beta\in \NN\setminus \Omega:v_{\beta} \underset{(\FF_2)}{=} v_{\alpha_i} + v_{\alpha_j} + v_{\alpha_k}, \exists \alpha_i,\alpha_j,\alpha_k\in \cup_{r=0}^{\ell-1}B_{r}\}$$
    for $\ell\geq 1$.
    We will show by induction on $s$ that \eqref{eq:prop_weighted} holds for any $\beta \in B_s$.
    The base case, $s=0$, is trivial from the definition of $I_2$. Suppose that \eqref{eq:prop_weighted} holds for $s\leq \ell$. Consider the case when $s=\ell+1$.
    For any $\beta\in B_{\ell+1}$, there exists $\alpha_i,\alpha_j,\alpha_k \in  \cup_{r=0}^{\ell}B_{r}$ such that 
    $v_{\beta} \underset{(\FF_2)}{=} v_{\alpha_i} + v_{\alpha_j} + v_{\alpha_k}$. There are two possibilities, one is that $\{\beta,\alpha_i;\alpha_j,\alpha_k\}$ is a square, and the other one is that $\{\beta,\alpha_i,\alpha_j,\alpha_k\}$ is a generalized square.
     
    Firstly, suppose $\{\beta, \alpha_i; \alpha_j, \alpha_k\}$ is a square. In this case, we have $h_{\beta,\alpha_i}- h_{\alpha_j, \alpha_k} \in J_2$, and let $\tau(\alpha,\beta;\delta) = ( T_\alpha  +\delta)/T_\beta$. Observe that
     
     \begin{equation}\label{eq:append1}
     \begin{split}
         g_\beta +& \big(\tau(\beta,\alpha_i;c_\beta)\big)^2
         g_{\alpha_i} + g_{\alpha_j}+ \big(\tau(\beta,\alpha_i;c_\beta)\big)^2 g_{\alpha_k}\\
         =& - 2c_\beta f_\beta+ \big(f_\beta - \tau(\beta,\alpha_i;c_\beta)f_{\alpha_i}\big)^2 + \big(f_{\alpha_j} + \tau(\beta,\alpha_i;c_\beta)f_{\alpha_k}\big)^2\\
          &+\hat\tau_{\alpha_j} f_{\alpha_j} +\hat\tau_{\alpha_k} f_{\alpha_k} + \hat\tau_{\alpha_i}f_{\alpha_i} \bmod J_2.
         \end{split}
     \end{equation}
     where \(\hat\tau_{\alpha_i} = 2c_\beta\tau(\beta,\alpha_i;c_\beta)\), \(\hat \tau_{\alpha_j} = 2 (\hat{T}_{\alpha_j}+\hat{T}_\beta+c_\beta)\) and \(\hat\tau_{\alpha_k} = 2 \tau(\beta,\alpha_i;c_\beta)(\hat{T}_{\alpha_j}+\hat{T}_\beta+c_\beta)\).
     From \eqref{eq:prop_weighted}, for each $\alpha_q \in \cup_{r=0}^\ell B_r$ and constant $\hat\tau_{\alpha_q}$, $q=i,j,k$, there exists $S_{\alpha_q} \subseteq \NN$ and weights $\{w^{(\alpha_q)}_\gamma\}_{\gamma\in S_{\alpha_q}}$ such that
     \begin{equation}\label{eq:append6}
          P_{\alpha_q} := g_{\alpha_q} + \sum_{\gamma\in S_{\alpha_q}} w^{(\alpha_q)}_\gamma g_\gamma = -2\hat\tau_{\alpha_q}f_{\alpha_q} + \tilde p_{\alpha_q} \bmod (I_2 + J_2),
     \end{equation}
     where $\tilde p_{\alpha_q} \in \Sigma_{N,2}$.
     Then we have
     \begin{equation}\label{eq:append3}
          2\hat\tau_{\alpha_q}f_{\alpha_q} + P_{\alpha_q} = \tilde{p}_{\alpha_q}\in\Sigma_{N,2}.
     \end{equation}
     which implies that the last three terms on the right-hand side (RHS) of \eqref{eq:append1} can be combined with $P_{\alpha_q},q=i,j,k$ to form SOS polynomials. Hence,
     the RHS of \eqref{eq:append1} can be simplified as
     \begin{equation}\label{eq:append7}
         - 2c_\beta f_\beta + \tilde p_\beta \bmod (I_2 + J_2),~~\text{where $\tilde p_\beta\in \Sigma_{N,2}$}.
     \end{equation}
     which implies that \eqref{eq:prop_weighted} holds.
   
     Next, we consider the case that $\{\beta,\alpha_i,\alpha_j,\alpha_k\}$ is a GS with $\alpha_i,\alpha_j,\alpha_k \in \Omega$.
     Denote $d_H(v_1, v_2)$ as the Hamming distance.
     Without loss of generality, suppose that 
     \[
     d_H(v_\beta,v_{\alpha_i}) \geq \max\{d_H(v_\beta,v_{\alpha_j}), d_H(v_\beta,v_{\alpha_k})\}
     \]
     i.e, $\alpha_i$ is the tuple which shares the least number of common indices with $\beta$.
     Then $\alpha_i$ and $\beta$ at least have two distinct elements, otherwise there will be a contradiction with that $\beta, \alpha_i, \alpha_j$ and $\alpha_k$ are distinct elements satisfying the GS-propagation rule.
      Assume $\beta$ and $\alpha_i$ have $m$ distinct indices on the first $m$ slots where $2\leq m \leq d$ and the remaining slots are the same.
      Then $\alpha_j$ and $\alpha_k$ also have distinct elements on the first $m$ slots and the remaining slots are the same.
      Hence,
      \begin{gather*}
      \beta = (i_1, i_2, \dots, i_m, i_{m+1}\dots, i_d),\quad
      \alpha_i = (j_1, j_2, \dots, j_m, i_{m+1} \dots, i_d),\\
      \alpha_j = (\bar i_1, \bar i_2, \dots, \bar i_m, j_{m+1}, \dots, j_d), \quad
      \alpha_k = (\bar j_1, \bar j_2, \dots, \bar j_{m}, j_{m+1}, \dots, j_d)
      \end{gather*}
      where $(\bar i_q, \bar j_q)$ equals either $(i_q,j_q)$ or $(j_q,i_q)$ for $1\leq q\leq m$.
      Consider the tuples
     \begin{align*}
         \gamma_1 &= (\bar i_1, \bar i_2, \dots, \bar i_m,i_{m+1}\dots, i_d),\\ 
         \gamma_2 &= (\bar j_1, \bar j_2, \dots, \bar j_m,i_{m+1}\dots, i_d),\\  
         \gamma_3 &= (j_1, j_2, \dots, j_m,j_{m+1}\dots, i_d).
     \end{align*}
     It follows that \[
     \{h_{\beta,\alpha_i} - h_{\gamma_1, \gamma_2},~ h_{\alpha_i,\alpha_j} - h_{\gamma_1,\gamma_3}, ~
     h_{\alpha_i,\alpha_k} - h_{\gamma_2, \gamma_3}\}\subset J_2.
     \]
     Recall $t_{\alpha,\beta}:=T_\alpha/T_\beta$, and notice that
     \begin{equation}\label{eq:append2}
     \begin{split}
         g_\beta +&
         \big(2(\tau(\beta,\alpha_i;c_\beta))^2+1\big)g_{\alpha_i}+
         \tau^2_{\alpha_j}g_{\alpha_j} + \tau^2_{\alpha_k}g_{\alpha_k}+ g_{\gamma_1} + \tau_{\gamma_2}^2 g_{\gamma_2}^2+\tau_{\gamma_3}^2g_{\gamma_3}^2 \\
         =& - 2c_\beta f_\beta+\big(f_\beta - \tau(\beta,\alpha_i;c_\beta) f_{\alpha_i}\big)^2 + \big(f_{\gamma_1} + \tau(\beta,\alpha_i;c_\beta) f_{\gamma_2} - \tau_{\gamma_3}f_{\gamma_3}\big)^2\\ & + \big(f_{\alpha_i}+ \tau_{\gamma_3} f_{\alpha_j}\big)^2+ \big(\tau(\beta,\alpha_i;c_\beta) f_{\alpha_i} + \tau_{\gamma_3}f_{\alpha_k} \big)^2 \\
         & + 2\tilde\tau_{\alpha_i}f_{\alpha_i}  + 2\tilde\tau_{\alpha_j}f_{\alpha_j}+2\tilde\tau_{\alpha_k}f_{\alpha_k}\bmod J_2
         \end{split}
     \end{equation}
     where $\tau_{\gamma_2} = \tau(\beta,\alpha_i;c_\beta) $, $\tau_{\gamma_3} = t_{\gamma_1,\gamma_3} +\tau(\beta,\alpha_i;c_\beta) t_{\gamma_2,\gamma_3}$, $\tau_{\alpha_j} = \tau_{\alpha_k} =\tau_{\gamma_3}^2$, $\tilde\tau_{\alpha_j} = (T_{\alpha_i} + \tau_{\gamma_3} T_{\alpha_j})\tau_{\gamma_3}$, $\tilde\tau_{\alpha_k} = (\tau_{\alpha_i}T_{\alpha_i} + \tau_{\gamma_k}T_{\alpha_k})\tau_{\gamma_3}$, and $\tilde\tau_{\alpha_i} = \tilde\tau_{\alpha_j}+\tilde\tau_{\alpha_k}$. 

     From \eqref{eq:prop_weighted}, for each $\alpha_q \in \cup_{r=0}^\ell B_r$ and constant $\tilde \tau_{\alpha_q}$, $q=i,j,k$, there exists $\tilde S_{\alpha_q} \subseteq \NN$ and weights $\{\tilde w^{(\alpha_q)}_\gamma\}_{\gamma\in \tilde S_{\alpha_q}}$ to construct polynomial $\tilde P_{\alpha_q}$ similarly as in \eqref{eq:append6}. Then the last three terms on the right-hand side (RHS) of \eqref{eq:append2} can be combined with $\tilde P_{\alpha_q},q=i,j,k$ to form SOS polynomials. Hence, the RHS of \eqref{eq:append2} can be simplified as
     \begin{equation}
         - 2c_\beta f_\beta + \tilde p_\beta' \bmod (I_2 + J_2),~~\text{where $\tilde p_\beta'\in \Sigma_{N,2}$}.
     \end{equation}
     which implies that \eqref{eq:prop_weighted} holds. Thus, \eqref{eq:prop_weighted} holds for any $\beta\in B_{\ell+1}$. Since GS-propagation condition holds, it is true that \eqref{eq:prop_weighted} holds for any $\beta \in \NN\setminus \Omega$.
     
     In addition, if $\Omega$ satisfies S-propagation, suppose that $S_\beta\subseteq \cup_{r=0}^{\ell-1} \tilde B_r$ for all $\beta\in\tilde B_{\ell}$. From the choice of $\alpha_q$ and $S_{\alpha_q}, q=i,j,k$ as shown in \eqref{eq:append1} to \eqref{eq:append7} and the S-propagation rule, it follows that $S_\beta\subseteq \cup_{r=0}^{\ell} \tilde B_r$ for all $\beta\in\tilde B_{\ell+1}$.
     
     It remains to prove the base case, that is, for $\beta \in \tilde B_1$, then $S_\beta \in \mathcal{P}_{SR}(\Omega)$. Because $\mathcal{P}_{SR}(\Omega)$ itself is a propagation set, we will prove the base case by induction. Let $\mathcal{D}_0 = \Omega$, and $\mathcal{D}_\ell = \{\beta \in \NN\setminus \Omega: \exists~\alpha_i,\alpha_j\in \Omega, \alpha_k\in \cup_{r=0}^{\ell-1}D_r\text{~s.t.~} \{\beta,\alpha_i;\alpha_j,\alpha_k\} \text{~is a square}\}$. 
      For $\beta \in \mathcal{D}_1$, there exists $\alpha_i,\alpha_j,\alpha_k\in \Omega$ such that $\{\beta,\alpha_i;\alpha_j,\alpha_k\}$ is a square, and this case is trivial from \eqref{eq:append1}, because for the last three terms on the RHS of \eqref{eq:append1}, we have $f_{\alpha_q} = 0 \bmod I_2,q=i,j,k$. Suppose that for $\beta\in\mathcal{D}_\ell$, we have $S_\beta \in \cup_{r=0}^{\ell-1}\mathcal{D}_r$. 

      Since SR-propagation is a special case of S-propagation, following the same inductive step from $\tilde B_\ell$ to $\tilde B_{\ell+1}$, we can show that for any $\beta \in \mathcal{D}_{\ell+1}$, $S_\beta \in \cup_{r=0}^{\ell}\mathcal{D}_r$. 
      There exists a positive integer $d_s$ such that $\mathcal{P}_{SR}(\Omega) = \cup_{r=0}^{d_s}\mathcal{D}_r$. Thus, we have shown the base case, i.e., for $\beta \in \tilde B_1$, then $S_\beta \in \mathcal{P}_{SR}(\Omega)$.
\end{proof}

Then we are ready to prove \Cref{thm:tight_weighted_sdp}.

\begin{proof}[Proof of \Cref{thm:tight_weighted_sdp}]
     From \Cref{lem:tight_weighted_sdp}, if the GS-propagation condition holds, for any $\beta\in \NN\setminus \Omega$, let $c_\beta=T_\beta$, there is a set of multi-indices $S_\beta$ and set of nonnegative weights $\{w^{(\beta)}_\gamma\}_{\gamma\in S_\beta}$ such that
     \[
    (x_\beta^2 - T_{\beta}^2) + \sum_{\gamma\in S_\beta} w^{(\beta)}_\gamma (x_{\gamma}^2 - T_\gamma^2) = -2T_\beta(x_\beta - T_{\beta}) + \tilde p_\beta \bmod (I_2 + J_2).
    \]
    Observe that $(x_\beta - T_{\beta})^2 = (x_\beta^2 - T_{\beta}^2) - 2T_\beta(x_\beta - T_{\beta})$. Then it can be shown that
    \begin{equation}\label{eq:gs_single_sos}
        2(x_\beta^2 - T_{\beta}^2) + \sum_{\gamma\in S_\beta} w^{(\beta)}_\gamma (x_{\gamma}^2 - T_\gamma^2) = (x_\beta - T_{\beta})^2 + \tilde p_\beta \bmod (I_2 + J_2). 
    \end{equation}
    Applying \eqref{eq:gs_single_sos} for all $\beta\in\NN\setminus\Omega$, it follows that there exists a nonnegative weight vector $w$ such that 
    \begin{equation}\label{eq:gs_all_sos}
        \sum_{\beta\in \NN\setminus\Omega} w_\beta g_\beta = \sum_{\beta\in \NN} w_\beta g_\beta = \sum_{\beta\in \NN} f^2_\beta + \hat p = \tilde \sigma \bmod (I_2 + J_2)
    \end{equation}
    where $\hat p,\tilde \sigma\in \Sigma_{N,2}$. Similarly to \Cref{lem:sdp2}, we have that the associated Gram matrix of $\tilde \sigma$ is of rank $N$, and $[1;\bar x]$ is in the kernel of this Gram matrix.
    Following the same analysis as in \Cref{thm:tightsdp2},
    it can be shown that
    \eqref{eq:weighted_sdp} is tight for the given nonnegative weight vector~$w$.
\end{proof}

\subsection{Proof of \Cref{thm:tight_sdp_square}}\label{append:thm:tight_sdp_square}

\begin{proof}
     For any $\beta\in \NN\setminus \Omega$, let $c_\beta=T_\beta$. Following the same strategy to construct \eqref{eq:gs_single_sos} as in \Cref{thm:tight_weighted_sdp}, for any $\beta \in\tilde B_{\ell+1}$, by \Cref{lem:tight_weighted_sdp}, there exists a set $S_{\beta}\subseteq \cup_{r=0}^\ell \tilde B_r$ and weights $\{ w^{(\beta)}_\gamma\}_{\gamma\in S_{\beta}}$ such that 
     \begin{equation}\label{eq:append_sos}
     \vartheta^{\ell+1}\big(2g_{\beta}^2 + \sum_{\gamma\in S_{\beta}} w^{(\beta)}_\gamma g_{\gamma}^2 \big) =  \vartheta^{\ell+1}(f_{\beta}^2 + \tilde p_{\beta}) \bmod (I_2 + J_2). 
     \end{equation}
     
     From the way of assigning weights in \Cref{alg:weights}, when $\vartheta
     $ is sufficiently small, that is, $\vartheta\in (0, t)$ for some $t\in(0,1)$, the weights $\vartheta^{\ell+1}$ for the elements in $\tilde B_{\ell+1}$ are much smaller than the weights $\vartheta^{s}$ in $\tilde B_{s}, s\leq \ell$. Therefore, for $\beta\in\tilde B_{\ell+1}$, the weights $\{\vartheta^{\ell+1}w^{(\beta)}_\gamma\}_{\gamma\in\tilde S_{(\beta)}}$ required to form SOS polynomials in \eqref{eq:append_sos} are significantly smaller that the assigned weights (which are at least $\vartheta^{\ell}$) of $g_\gamma \text{~for~}\gamma\in S_\beta \subseteq\big(\cup_{r=0}^\ell \tilde B_r\big)\cup\mathcal{P}_{SR} $. 
     Since S-propagation holds, there exists a postive integer $m$ such that
    \[  
    \cup_{r=1}^m \tilde B_r = \NN\setminus \mathcal{P}_{SR}(\Omega).
    \]
    which implies that \eqref{eq:append_sos} holds for any $\beta\in\NN\setminus \mathcal{P}_{SR}(\Omega)$.

    It remains to address the case when $\beta\in\mathcal{P}_{SR}(\Omega)$.
    Recall the initial weight assigned to each $\beta\in\mathcal{P}_{SR}(\Omega)$ is set to 1. Under the assumption that $\vartheta$ is sufficiently small, only a negligible portion of the weight for $\beta\in\mathcal{P}_{SR}(\Omega)$ is redistributed to combine with elements in $\NN\setminus\mathcal{P}_{SR}(\Omega)$ for the construction of SOS polynomials. Hence, for each $\beta\in\mathcal{P}_{SR}(\Omega)$, its remaining weight $\iota_\beta$ is close to 1. From \Cref{lem:sdp2}, it follows that
    \[
    \sum_{\beta\in\mathcal{P}_{SR}(\Omega)} \iota_\beta g_\beta = \sigma_w \bmod (I_2 + J_2)
    \]
    where $\sigma_w\in\Sigma_{N,2}$ and the rank of the Gram matrix associated with $\sigma_w$ equals the cardinality $|\mathcal{P}_{SR}(\Omega)|$.
    Following the same analysis as in \Cref{thm:tight_weighted_sdp}, the proof is complete.
\end{proof}
\subsection{Proof of \Cref{lem:poly_error}}\label{append:lem:poly_error}
\begin{proof}
    Since $p \in I_2$, then $p = \sum_{\alpha\in\Omega} (\nu_\alpha f_\alpha + \lambda_\alpha g_\alpha)$. Let
    \begin{equation}\label{eq:8}
         p^\epsilon :=  \sum_{\alpha\in\Omega}(\nu_\alpha f^\epsilon_\alpha + \lambda_\alpha g^\epsilon_\alpha) = p - \kappa^\epsilon_{\lambda,\nu},
    \end{equation}
    where $\kappa^\epsilon_{\lambda,\nu}=\sum_{\alpha\in \Omega}(\lambda_\alpha \epsilon^2_\alpha +2\lambda_\alpha\hat{T}_\alpha\epsilon_\alpha + \nu_\alpha\epsilon_\alpha)$ = $O(\|\epsilon\|_\infty)$. Observe that
    $(f^\epsilon_\alpha)^2 = g^\epsilon_\alpha - 2f^\epsilon_\alpha$. It follows that
    \begin{equation}\label{eq:9}
    \begin{split}
          p^\epsilon + C\sum_{\alpha\in\Omega}(f^\epsilon_{\alpha})^2 =& \sum_{\alpha\in\Omega} \Big((C+\lambda_\alpha)g^\epsilon_\alpha - (2C-\nu_\alpha)f^\epsilon_\alpha\Big)\\
          =&\sum_{\alpha\in\Omega}\Big((C+\lambda_\alpha)\big(x_\alpha - s_{\alpha}^\epsilon(C)\big)^2 + \bar{s}^\epsilon_\alpha(C)\Big),
    \end{split}
    \end{equation}
    where $s^\epsilon_\alpha = \frac{2CT^\epsilon_\alpha - \nu_\alpha}{2(C+\lambda_\alpha)}$, and $\bar s^\epsilon_\alpha = -\frac{4\lambda_\alpha^2(T^\epsilon_\alpha)^2 + 4\lambda_\alpha\bar \lambda_\alpha T^\epsilon_\alpha + \nu_\alpha^2}{4(C+\lambda_\alpha)} = O(1/C)$.
    Then
    \begin{equation}\label{eq:10}
    \begin{split}
         p + C\sum_{\alpha\in\Omega}(f^\epsilon_\alpha)^2 
         =& (p^\epsilon + \kappa^\epsilon_{\lambda,\nu})+ C\sum_{\alpha\in\Omega}(f^\epsilon_\alpha)^2 \\
         =&  \tilde p^\epsilon + \sum_{\alpha\in\Omega}\bar s^\epsilon_\alpha + O(\|\epsilon\|_\infty)\\
         =& \tilde p^\epsilon + O(1/C) + O(\|\epsilon\|_\infty)
    \end{split}
    \end{equation}
    where $\tilde p^\epsilon := \sum_{\alpha\in\Omega}(C+\lambda_\alpha)(x_\alpha - s_{\alpha}^\epsilon(C))^2\in\Sigma_{N,2} $.
    The first equality in \eqref{eq:10} follows from \eqref{eq:8}, and the second equality is from \eqref{eq:9}.
     Plugging in $x_\alpha = \hat{T}_\alpha, \forall \alpha\in\Omega$ the left-hand side (LHS) of \eqref{eq:10} is 
    \begin{equation}\label{eq:11}
    \begin{split}
    p = 0 \quad\text{and}\quad  C\sum_{\alpha\in\Omega}(f^\epsilon_\alpha)^2= C\|\epsilon\|_\infty^2
    \end{split}
    \end{equation}
    From \eqref{eq:10} and \eqref{eq:11},
    plugging in
    $x_\alpha = \hat T_\alpha, \forall\alpha\in\Omega$
    at the polynomial $\tilde p^\epsilon$
    is a scalar of magnitude $O(\|\epsilon\|_\infty + \|\epsilon\|_\infty^2C + 1/C)$. 
\end{proof}
\subsection{Proof of \Cref{lem:dual_bound}}\label{append:lem:dual_bound}
\begin{proof}
From \Cref{thm:tightsdp2}, there exists $\bar p=\sum_{\alpha\in\Omega}(\bar \lambda_\alpha g_\alpha + \bar\nu_\alpha f_\alpha)\in I_2$ and $\sigma \in \Sigma_{N,2}$ such that
\begin{equation}\label{eq:sos_noise}
    \sum_{\beta\in \NN} g_\beta - \bar p = \sum_{\beta\in \NN} g_\beta - \sum_{\alpha\in \Omega} \Big(\bar\lambda_\alpha g_\alpha + \bar\nu_\alpha f_\alpha\Big)
    = \sigma\bmod J_2,
\end{equation}
and the Gram matrix $\bar Q$ of $\sigma\in\Sigma_{N,2}$ has corank 1. We also have
\[
\bar Q \bullet \boldsymbol{\bar X} = 0.
\]
From \eqref{eq:9} and \eqref{eq:sos_noise}, it follows that
\begin{equation}\label{eq:noisy_sos}
    \begin{split}
        &\Big(\big(\sum_{\beta\in\NN} g_\beta\big) - \bar p\Big) + \Big(\bar p -\kappa^\epsilon_{\bar \lambda,\bar \nu} + C\sum_{\alpha\in\Omega}(f^\epsilon_{\alpha})^2- \sum_{\alpha\in\Omega}\bar{s}^\epsilon_\alpha(C)\Big)\\
        =& \sigma + \tilde p^\epsilon \in \Sigma_{N,2}.
    \end{split}
\end{equation}
Also, when $x_\beta = T_\beta,\forall \beta\in\NN$, the value of the above polynomial is
$O(\|\epsilon\|_\infty + \|\epsilon\|_\infty^2C + 1/C)$. Let $\bar Q^\epsilon_c$ be the associated Gram matrix of the above SOS polynomial. In \eqref{eq:noisy_sos}, we have that $\tilde p^\epsilon \in \Sigma_{N,2}$. It follows that $\bar Q^\epsilon_c\succeq \bar Q$. And for $\boldsymbol{\bar X}$ and $\bar Q^\epsilon_c$, the duality gap satisfies the following inequality,
\begin{equation}\label{eq:01}
\bar Q_c^\epsilon \bullet \boldsymbol{\bar X} \leq O(\|\epsilon\|_\infty + \|\epsilon\|_\infty^2C + 1/C),
\end{equation}
which further implies that
\begin{equation}\label{eq:dual_gap}
    \bar Q_c^\epsilon\bullet \boldsymbol{X_c^\epsilon} \leq \bar Q_c^\epsilon \bullet \boldsymbol{\bar X}\leq O(\|\epsilon\|_\infty + \|\epsilon\|_\infty^2C + 1/C).\end{equation}
where the first inequality follows from the primal optimality of $\boldsymbol{X_c^\epsilon}$. 
This proves~\eqref{eq:dual_bound}.
\end{proof}

\section{Custom SDP solver}\label{experiments:custom_solver}

In order to solve completion problems for larger tensors,
we develop a custom SDP solver using ideas similar to those in \cite{cosse2021stable}.
We focus here on problem \eqref{eq:sdp-p}.
The main difficulty is that the SDP has a very large number of scalar variables and constraints.
We make use of three techniques to overcome this:
hierarchical low-rank factorization,
averaging the symmetry constraints,
and subsampling the constraints.
We proceed to explain these techniques, and then we provide our overall method.

\subsubsection*{Hierarchical Low-Rank Factorization}
We apply the low-rank Burer-Monteiro method \cite{burer2005local}.
Concretely, we assume that $\boldsymbol{X}$ has rank $\leq k$,
so we may factorize it as
\(\boldsymbol{X} = YY^T,\)
where $Y = [t; \Tilde{Y}]\in \RR^{(N+1)\times k}$.
Each column of $\tilde Y$ has length~$N$,
so we may view it as the vectorization of a tensor 
$\tilde T_{(i)} \in \RR^{\NN}$,
i.e.,
\( \Tilde{Y}_{(i)} = \vct(\tilde T_{(i)}). \)
We now assume that these tensors have rank $\leq q$, 
so we may write them as $\tilde T_{(i)} = \sum_{j=1}^q \tilde u^{(1)}_{(i,j)} \otimes \dots \otimes \tilde u^{(d)}_{(i,j)}$.
In this way, the original convex problem in terms of $\boldsymbol{X}$
becomes a nonconvex problem in variables $t \in \RR^k$ and $\tilde u_{(i,j)}^{(\ell)}\in \RR^{n_\ell}$ for $i \in [k], j\in [q], \ell \in [d]$.
This reduces the number of variables from $O(N^2)$ to $O(n k q)$.

\subsubsection*{Averaging symmetry constraints}

There is a very large number of constraints of the form
$X_{\alpha_1,\alpha_2} = X_{\alpha_3,\alpha_4}$;
we call them symmetric constraints.
We average some of these constraints together in order to reduce the number of constraints.
Given a pair of tuples $(\alpha, \beta) \in \NN^2$ and some $i \in [d]$,
let $\pi^i(\alpha, \beta) \in \NN^2$ be the pair of tuples obtained by switching the $i$-th index of $\alpha,\beta$.
Instead of the $|\mathcal{A}|$ symmetric constraints,
we use the averaged equations
\begin{equation}\label{eq:relaxed_sym_cons}
    X_{\alpha,\beta} = \frac{1}{d} \sum_{i=1}^d X_{\pi^i(\alpha,\beta)},
    \quad\forall (\alpha, \beta) \in \NN^2, i \in [d]
\end{equation}
We point out that in \cite{cosse2021stable} they also average the constraints, though differently.
They focus exclusively on the case $d=2$,
and they only use either $O\big(n^2\big)$ or $O(n)$ averaged symmetric constraints.
In contrast, we use many more constraints, $O(N^2)$, as this leads to better practical results.

\subsubsection*{Subsampling the constraints}
The number of constraints is still quite large, even after averaging the symmetry constraints.
To further reduce the computational cost, we randomly select batches of pairs $(\alpha,\beta) \in \NN^2$,
and we only use the constraints in \eqref{eq:relaxed_sym_cons} corresponding to such pairs.
In addition, we also subsample the observations and consider only a small batch of terms
$X_{\alpha,\alpha} -2T^\epsilon_\alpha x_\alpha + ( T^\epsilon_\alpha)^2$ in the objective.
By doing this, the number of constraints and the number of terms in the objective is decreased to the batch size.
Our overall method requires solving multiple subproblems
and we resample the batches in each such subproblem.
By resampling we ensure that all the observations are used in the overall method,
even though each subproblem only uses a subset of them.

\subsubsection*{Overall method}

The starting point for our method is the augmented Lagrangian method (ALM).
The Lagrangian function for \eqref{eq:sdp-p} is
\begin{align*}
L_\rho(X,\lambda) = \tr(X)
  & + \frac{\rho}2\sum_{\alpha \in \Omega}
   (X_{\alpha,\alpha } -2\hat{T}_\alpha x_\alpha + \hat T_\alpha^2)\\
  &+ \!\!\!\!\! \sum_{\{\alpha_i,\alpha_j;\alpha_k,\alpha_\ell\}\in\mathcal{A}} \!\!\!\!\!
  \bigl(
  \mu_{ijk\ell} (X_{\alpha_i,\alpha_j} \!-\!  X_{\alpha_k,\alpha_\ell})
  +  \frac{\rho}{2}(X_{\alpha_i,\alpha_j} \!-\!  X_{\alpha_k,\alpha_\ell})^2
  \bigr)
\end{align*}
where $\mu$ is the Lagrange multipliers vector, and $\rho$ is the penalty parameter.
We slightly modify the objective by replacing the term
\(
   X_{\alpha,\alpha } -2\hat{T}_\alpha x_\alpha + \hat T_\alpha^2
\)
with
\(
   (X_{\alpha,\alpha } {-} \hat{T}_\alpha)^2 + (x_\alpha {-} \hat T_\alpha)^2,
\)
as this leads to a better performance in practice.
Let $B^{(1)} \subset \Omega$ and $B^{(2)} \subset \Omega\times \NN$ be the random batches for the subsampling.
Using constraint averaging and subsampling, we obtain the modified Lagrangian
\begin{align*}
\tilde L_\rho^{B^{(1)},B^{(2)}}(X,\lambda) &= \tr(X)
   + \sum_{\alpha \in B^{(1)}}\frac{\rho}2\sum_{\alpha \in \Omega}
   \Big((X_{\alpha,\alpha } {-} \hat{T}_\alpha)^2 + (x_\alpha {-} \hat T_\alpha)^2\Big)
 \\
  &+ \sum_{(\alpha,\beta) \in B^{(2)}}\Big(\lambda_{\alpha,\beta}
  (avg(X_{\alpha,\beta})  \!-\!  X_{\alpha,\beta}) +\frac{\rho}2(avg(X_{\alpha,\beta})  \!-\! X_{\alpha,\beta})^2\Big),
\end{align*}
where $avg(X_{\alpha,\beta}) = \frac{1}{d} \sum_{i=1}^d X_{\pi^i(\alpha,\beta)}.$ 
We also apply low-rank hierarchical decomposition,
to obtain a nonconvex problem in $O(n k q)$ variables.

The overall method follows the usual ALM framework,
iteratively solving the modified Lagrangian for some fixed multipliers,
and then updating the multipliers.
However, in each subproblem we may use different batches $B^{(1)},B^{(2)}$, obtaining a new modified Lagrangian.

We implemented our method in Julia, using NLopt for solving the nonconvex subproblems.
In our experiments,
the initial point was obtained by drawing independent entries from $\mathcal{N}(0,1)$.
We set the penalty parameter $\rho=10$,
the parameters $k = q = 2$ for the low-rank decompositions,
and we sample uniformly random batches with $|B^{(1)}| = 150$ and $|B^{(2)}| = 300$.
For each random batch, we applied 25 rounds of~ALM.
We used 20 random batches, for a total of 5000 rounds of~ALM.

\end{appendices}

\end{document}